\RequirePackage{fix-cm}
\documentclass{article}
\usepackage[top=1in, bottom=1.25in, left=1.25in, right=1.25in]{geometry}

\usepackage{graphicx}
\usepackage{amssymb,amsmath,amsthm,graphicx,comment}
\usepackage{dsfont,authblk}
\usepackage{mathtools}

\usepackage{float}
\usepackage{cite}
\usepackage{algorithm}
\usepackage{algorithmic}
\usepackage[dvipsnames]{xcolor}
\usepackage{todonotes}
\usepackage{caption}
\usepackage{subfigure}

\newtheorem{theorem}{Theorem}[section]  

\newtheorem{proposition}[theorem]{Proposition}
\newtheorem{lemma}[theorem]{Lemma}
\newtheorem{remark}[theorem]{Remark}

\newcommand{\beqa}{\begin{eqnarray}}
\newcommand{\eeqa}{\end{eqnarray}}
\newcommand{\beqas}{\begin{eqnarray*}}
\newcommand{\eeqas}{\end{eqnarray*}}
\newcommand{\beq}{\begin{equation}}
\newcommand{\eeq}{\end{equation}}
\newcommand{\beqs}{\begin{equation*}}
\newcommand{\eeqs}{\end{equation*}}

\newcommand{\R}{\mathbb{R}}

\newcommand{\ones}{{\mathds{1}}}

\newcommand{\bfone}{\mathbf{1}} 
\newcommand{\ddd}{\delta}

\newcommand{\norm}[1]{\left|\left|#1\right|\right|}
\newcommand{\tr}[1]{\text{tr}\left(#1\right)}
\newcommand{\trace}{\mbox{\rm trace}}

\newcommand{\hub}{\overline{h}}
\newcommand{\qlb}{\underline{q}}
\newcommand{\zub}{\overline{z}}
\newcommand{\yub}{\overline{y}}

\newcommand{\cdk}{B_{\text{CD}\mbox{-}k}}
\newcommand{\cd}{B_{\text{CD}}}
\newcommand{\ccd}{B_{\text{CCD}}}
\newcommand{\ccdpi}{B_{\text{CCD}\mbox{-}\pi}}
\newcommand{\ccdpiell}{B_{\text{CCD}\mbox{-}\pi_\ell}}
\newcommand{\rcd}{B_{\text{RCD}\mbox{-}k}}
\newcommand{\rcdE}{B_{\text{RCD}}}
\newcommand{\rpcd}{B_{\text{RPCD}\mbox{-}\ell}}
\newcommand{\rpcdE}{B_{\text{RPCD}}}
\newcommand{\xc}{x_{\text{CCD}}}
\newcommand{\xcpi}{x_{\text{CCD}\mbox{-}\pi}}
\newcommand{\xr}{x_{\text{RCD}}}
\newcommand{\xp}{x_{\text{RPCD}}}

\newcommand{\xs}{{x^*}}

\newcommand{\bma}{\begin{math}}
\newcommand{\ema}{\end{math}}

\newcommand{\E}{\mathbb E}
\newcommand{\bigO}{{\cal O}}
\newcommand{\bigOt}{\widetilde{\cal O}}
\newcommand{\Ly}{{\cal I}}

\begin{document}

\title{Randomness and Permutations in Coordinate Descent Methods}



\author{Mert G\"urb\"uzbalaban\thanks{Department of Management Science and Information Systems, Rutgers University, 100 Rockafellar Road, Piscataway, NJ 08854. \texttt{mg1366@rutgers.edu}.}, 
Asuman Ozdaglar\thanks{Department of Electrical Engineering and Computer Science, Massachusetts Institute of Technology, 77 Massachusetts Avenue, Cambridge, Massachusetts 02139. \texttt{asuman@mit.edu}.}, \\
Nuri Denizcan Vanli\thanks{Department of Electrical Engineering and Computer Science, Massachusetts Institute of Technology, 77 Massachusetts Avenue, Cambridge, Massachusetts 02139. \texttt{denizcan@mit.edu}.}\, and 
Stephen J. Wright\thanks{Department of Computer Sciences and Wisconsin Institute for Discovery, University of Wisconsin - Madison, 1210 West Dayton Street, Madison, WI 53706. \texttt{swright@cs.wisc.edu}.}}

%
%


\maketitle

\begin{abstract}
We consider coordinate descent (CD) methods with exact line search on convex quadratic problems. Our main focus is to study the performance of the CD method that use random permutations in each epoch and compare it to the performance of the CD methods that use deterministic orders and random sampling with replacement. We focus on a class of convex quadratic problems with a diagonally dominant Hessian matrix, for which we show that using random permutations instead of random with-replacement sampling improves the performance of the CD method in the worst-case. Furthermore, we prove that as the Hessian matrix becomes more diagonally dominant, the performance improvement attained by using random permutations increases. We also show that for this problem class, using any fixed deterministic order yields a superior performance than using random permutations. We present detailed theoretical analyses with respect to three different convergence criteria that are used in the literature and support our theoretical results with numerical experiments.
\end{abstract}

\section{Introduction}\label{sec:main}
We consider coordinate descent (CD) methods for solving unconstrained optimization problems of the form 
\begin{equation} \label{eq:f}
	\min_{x \in \R^n} \, f(x),
\end{equation}
where $f:\R^n \to \R$ is smooth and convex. CD methods have a long history in optimization \cite{luo1992convergence,bertsekas1989parallel,ortega2000iterative} and have been used in many applications \cite{PassCode15,NesterovCD17,richtarik2016parallel,qin2013efficient,scutari14}. They have seen a resurgence of recent interest because of their scalability and desirable empirical performance in machine learning and large-scale data analysis \cite{Bertsekas15Book,Shi16Survey,Wright2015}.

CD methods are iterative algorithms that perform (approximate) global minimizations with respect to a single coordinate (or several coordinates in the case of block CD) at each iteration. Specifically, at iteration $k$, an index $i_k \in \{1,2,\dots,n\}$ is chosen and the decision variable is updated to approximately minimize the objective function in the $i_k$-th coordinate direction (or at least to produce a significant decrease in the objective) \cite{Bertsekas99nonlinear, Bertsekas15Book}. The steps of this method are summarized in Algorithm~\ref{alg:cd}, where $e_i = [0,\dots,0,1,0,\dots,0]^T$ is the $i$-th standard basis vector (with the $i$-th entry equal to one). At each iteration $k$, $i_k$-th coordinate of $x$ is selected and a step is taken along the negative gradient direction in this coordinate. The counter $k=\ell n+j$ keeps track of the total number of iterations consisting of outer iterations indexed by $\ell$ and inner iterations indexed by the counter $j$. Each outer iteration is called a ``cycle" or an ``epoch'' of the algorithm.

\begin{algorithm}
\begin{algorithmic}
\STATE Choose initial point $x^0 \in \R^n$
\FOR{$\ell=0,1,2,\dotsc$}
\FOR{$j=0,1,2,\dotsc,n-1$}
\STATE Set $k=\ell  n+j$
\STATE Choose index $i_k=i(\ell,j) \in \{1,2,\dotsc,n\}$
\STATE Choose stepsize $\alpha_k>0$
\STATE $x^{k+1} \leftarrow x^k - \alpha_k [\nabla f(x^k)]_{i_k} e_{i_k}$, where $[\nabla f(x^k)]_{i_k} = e_{i_k}^T \nabla f(x^k)$
\ENDFOR
\ENDFOR
\end{algorithmic}
\caption{Coordinate Descent (CD) \label{alg:cd}}
\end{algorithm}

CD methods use various schemes, both deterministic and stochastic, for choosing the coordinate $i_k$ to be updated at iteration $k$. Prominent schemes include the following.
\begin{itemize}
\item Cyclic CD (CCD): The index $i(\ell,j)$ is chosen in a cyclic fashion over the elements in the set $\{1,2,\dots,n\}$ satisfying $i(\ell,j)=j+1$.
\item Cyclic CD with a given order $\pi$ (CCD-$\pi$): A permutation $\pi$ of the set $\{1,2,\dots,n\}$ is selected. Then, the index $i(\ell,j)$ is chosen as the $(j+1)$-th element of $\pi$ for every epoch $\ell$. (CCD corresponds to the special case of $\pi=(1,2,\dots,n)$.)
\item Randomized CD (RCD): The index $i(\ell,j)$ is chosen randomly with replacement from the set $\{1,2,\dotsc,n\}$ with uniform probabilities (each index has the same probability of being chosen). This method is also known as the \emph{stochastic CD} method.
\item Random Permutations Cyclic CD (RPCD): At the beginning of each epoch $\ell$, a permutation of $\{1,2,\dotsc,n\}$ is chosen, denoted by $\pi_{\ell}$, uniformly at random over all permutations. Then, the index $i(\ell,j)$ is chosen as the $(j+1)$-th element of $\pi_{\ell}$. Each permutation $\pi_{\ell}$ is independent of the permutations used at all previous and later epochs. This approach amounts to sampling indices from the set $\{1,2,\dotsc,n\}$ without replacement for each epoch.
\end{itemize}
While our focus in this paper will be on CD methods with the aforementioned selection rules, we note that several other variants of CD methods have been studied in the literature, including the Gauss-Southwell rule \cite{nutini2015coordinate}, in which $i_k$ is selected in a greedy fashion to maximize $[\nabla f(x^k)]_i$, and versions of RCD \cite{nesterov2012efficiency}, in which $i_k$ is selected from a non-uniform distribution that may depend on the component-wise Lipschitz constants of $f$.

We are interested in the relative convergence behavior of these different variants of CD. While there have been some recent works that study and compare performances of CCD and RCD (for example, \cite{nesterov2012efficiency,WangLinCCD,TewariSIAM,beck2013convergence,sun2015improved,sun2016worst,ccd_vs_rcd}); with the exception of a few recent papers (which focus on special quadratic problems, see \cite{wrightRPCD15,wrightRPCD17}), there is limited understanding of the effects of random permutations in CD methods.

In this paper, we study convergence rate properties of RPCD for a special class of quadratic optimization problems with a diagonally dominant Hessian matrix, and compare its performance to that of RCD and CCD. Interest in RPCD is motivated by both empirical observations and practical implementation: In many machine learning applications, RPCD is observed numerically to outperform its with-replacement sampling counterpart RCD \cite{Needell2014a,Recht2012}. Moreover, without-replacement sampling-based algorithms (such as RPCD and random reshuffling \cite{mert_RandomReshuffling,Bertsekas15}) are often easier to implement efficiently than their with-replacement counterparts (such as RCD and stochastic gradient descent) \cite{Recht2012,wrightRPCD15} as it requires sequential data access, in contrast to the random data access required by with-replacement sampling (see e.g. \cite{shamir2016without,Bottou2012}).

We start by surveying briefly the existing results on the effects of random permutations for CD methods \cite{sun2016worst,wrightRPCD15,wrightRPCD17,oswald2017random}. Among these, Oswald and Zhou~\cite{oswald2017random} studies the effects of random permutations on the convergence rate of the successive over-relaxation (SOR) method (that is used to solve linear systems) and presents a convergence rate on the expected function value of the iterates generated by the SOR method. The CD method, when applied to quadratic minimization problems, is equivalent to the SOR method (applied to the linear system that represents the first-order optimality condition of the quadratic problem) when the relaxation parameter is chosen as $\omega=1$. Therefore, the convergence rate results in \cite{oswald2017random} readily extend for RPCD, when applied to quadratic problems. Sun and Ye \cite{sun2016worst} construct a quadratic problem, for which CCD requires $\bigO(n^2)$ times more iterations compared to RCD in order to achieve an $\epsilon$-optimal solution (that is, a point $x^k$ that satisfies $\E f(x^k) - f(x^*) \leq \epsilon$). For this problem, they also show that the distance of the iterates (to the optimal solution) for CCD decays $\bigO(n^2)$ times slower than the distance of the expected iterates for RPCD and RCD. Lee and Wright \cite{wrightRPCD15} consider the same problem and present that the expected function values of RPCD and RCD decay with similar rates, while the asymptotic convergence rate of RPCD is shown to be slightly better than for RCD. In a following paper \cite{wrightRPCD17}, the results in \cite{wrightRPCD15} are generalized to a larger class of quadratic problems through a more elaborate analysis. 

Our main results provide convergence rate comparisons with respect to various criteria between RPCD, RCD, and CCD for a class of strongly convex quadratic optimization problems with a diagonally dominant Hessian matrix. In particular, we first provide an exact worst-case convergence rate comparison between RPCD, RCD, and CCD in terms of the distance of the expected iterates to the optimal solution, as a function of a parameter that represents the extent of diagonal dominance of the Hessian matrix. Our results show that, on this problem, CCD is always faster than RPCD, which in turn is always faster than RCD. Furthermore, we show that the relative convergence rate of RPCD to RCD goes to infinity as the Hessian matrix becomes more diagonally dominant. On the other extreme, as the Hessian matrix becomes less diagonally dominant, the ratio of convergence rates converges to a value in $[3/2, \, e-1)$, with the upper bound $e-1$ achieved in the limit as $n \to \infty$. Our second set of results compares the convergence rates of RPCD and RCD with respect to two other criteria that are widely used in the literature: the expected distance of the iterates to the solution and the expected function values of the iterates. For these criteria, we show that RPCD is faster than RCD in terms of the tightest upper bounds we obtain, and the amount of improvement increases as the matrices become more diagonally dominant.

The organization of the paper is as follows. In Section~\ref{sec:prelim}, we discuss the CCD, RCD, and RPCD algorithms in more detail and describe the three criteria that are used for analyzing convergence throughout the paper. In Section~\ref{sec:prior}, we survey known results on the convergence rate of RPCD. We analyze the convergence rates of CCD, RCD, and RPCD with respect to the first convergence criterion in Section~\ref{ssec:lyapunov1} and the behavior of RCD and RPCD with respect to the second and third convergence criteria in Section~\ref{ssec:lyapunov2and3}. We validate our theoretical results via numerical experiments in Section~\ref{sec:experiments} and present conclusions in Section~\ref{sec:conclusion}.

\section{Preliminaries}\label{sec:prelim}
To study performance of different CD methods, we focus on the special case of problem \eqref{eq:f} when $f$ is a strongly convex quadratic function:\footnote{The results can be generalized for quadratic functions of the form $f(x) = \frac12 x^TAx - b^Tx$; however, for simplicity and compatibility with the earlier results in the literature, we consider the case $b=0$.}
\begin{equation} 
f(x) = \frac12 x^TAx, \label{eq:obj1}
\end{equation}
where $A$ is a positive definite matrix. We denote its extreme eigenvalues by
\beq \label{def-muL}
\mu := \lambda_{\min}(A) >0, \quad
L := \lambda_{\max}(A),
\eeq
and note that $\mu$ is the modulus of convexity for $f$, while $L$ is the Lipschitz constant for $\nabla f$. The problem \eqref{eq:f} has a unique solution $x^*= 0$ with optimal value $f(\xs)=0$.

In the remainder of this section, we derive explicit formulas for the iterates of different variants of CD applied to \eqref{eq:f} (in terms of matrix operators representing each epoch)  and then introduce different convergence criteria for these variants. We show how asymptotic convergence rates can be characterized in terms of the spectral properties of $A$ and the matrix operators for each epoch.

\subsection{CD Methods}
In this section, we describe the variants of the CD method (in particular, CCD,  CCD-$\pi$, RCD, and RPCD) when applied to the quadratic problem in \eqref{eq:obj1}. The CD method (cf. Algorithm \ref{alg:cd}) with exact line search has the following update rule at each iteration
\begin{equation} \label{eq:cd-exact}
	x^{k+1} = x^k - \frac{1}{A_{i_k i_k}} (Ax^k)_{i_k} e_{i_k},
\end{equation}
where the update coordinate $i_k$ is determined according to one of the schemes mentioned above.

For the CCD algorithm, each coordinate is processed in a round-robin fashion using the standard cyclic order $(1,2,\dots,n)$. Denoting by $D$ the diagonal part of $A$ and by $-N$ the strictly lower triangular part of $A$, that is,
\begin{equation}
	A = D - N - N^T, \nonumber
\end{equation}
the evolution of the iterates over an epoch (of $n$ consecutive iterations) can be written as
\begin{equation}
	\xc^{(\ell+1)n} = \ccd \, \xc^{\ell n}, \quad \mbox{with} \quad \ccd = (D-N)^{-1}N^T, \label{eq:ccd}
\end{equation}
where $\ell$ denotes the epoch counter. Note that the update rule in \eqref{eq:ccd} is equivalent to one iteration of the Gauss-Seidel method applied to the first-order optimality condition of \eqref{eq:f}, which is the linear system $Ax=0$~\cite{Wright2015}.

For the CCD-$\pi$ algorithm, we let $P_\pi$ denote the permutation matrix corresponding to order $\pi$ and split the permuted Hessian matrix as follows:
\begin{equation} \label{eq:split}
    A_\pi = P_\pi^T A P_\pi = D_\pi - N_\pi  - N_\pi^T,
\end{equation}
where $-N_\pi$ is a strictly lower triangular matrix and $D_\pi$ is a diagonal matrix. Then, similar to \eqref{eq:ccd}, we have
\begin{equation}
	\xcpi^{(\ell+1)n} = \ccdpi \, \xcpi^{\ell n}, \quad \mbox{with} \quad \ccdpi = (D_\pi-N_\pi)^{-1}N_\pi^T. \label{eq:ccd-pi}
\end{equation}
Note that $\ccd$ and $\ccdpi$ are not symmetric matrices as the first column of both matrices are zero, whereas the first row contains nonzero entries.

For the RCD algorithm, the indices $i_k$ are chosen independently at random at each iteration $k$. Denoting by $\xr^k$ the $k$-th iterate generated by RCD, the update rule for RCD over a single iteration can be written as
\begin{equation}
	\xr^{k+1} = \rcd \, \xr^k, \quad \mbox{with} \quad \rcd = I - \frac{1}{A_{i_k i_k}} e_{i_k} e_{i_k}^T A. \label{eq:rcd-definition}
\end{equation}
The expectation of $\rcd$ with respect to the random variable $i_k$ is denoted as follows:
\begin{equation}\label{eq:Brcd}
\rcdE = \E_k \rcd,
\end{equation}
where we note that $\rcdE$ is a symmetric matrix, by symmetry of $A$ and uniform distribution of $i_k$.

For the RPCD algorithm, each coordinate is processed exactly once in each epoch according to a uniformly and independently chosen order. Recalling that $\pi_\ell$ denotes the permutation of coordinates used in epoch $\ell$ and using the iteration matrix corresponding to CCD-$\pi_\ell$ (see \eqref{eq:ccd-pi}), epoch $\ell$ of RPCD can be written as
\begin{equation}
	\xp^{(\ell+1)n} = \rpcd \, \xp^{\ell n}, \quad \mbox{with} \quad \rpcd = P_{\pi_\ell}  \ccdpiell P_{\pi_\ell}^T. \label{eq:rpcd}
\end{equation}
We introduce the following notation for the expected value of $\rpcd$ with respect to permutation $\pi_\ell$:
\begin{equation} \label{eq:Brpcd}
\rpcdE = \E_{\ell} \rpcd,
\end{equation}
where we note that $\rpcdE$ is a symmetric matrix since $\pi_\ell$ is chosen uniformly at random over all permutations (see Lemma \ref{lem:wright_lemma}).

\subsection{Convergence Rate Criteria}\label{ssec:rate_criteria}
We next discuss how to measure and compare the convergence rates of different variants of CD. Three different improvement sequences have been used to measure the performance of CD methods in the literature:
\begin{align*}
	& (i)   & \Ly_1(x_\text{CD}^k) & = \norm{\E x_\text{CD}^k-\xs}, & \text{(Distance of expected iterates)} \\
    & (ii)  & \Ly_2(x_\text{CD}^k) & = \E\norm{x_\text{CD}^k-\xs}^2, & \text{(Expected distance of iterates)} \\
    & (iii) & \Ly_3(x_\text{CD}^k) & = \E f(x_\text{CD}^k) - f(\xs). & \text{(Expected function value)} 
\end{align*}
(see e.g. \cite{sun2015improved,sun2016worst,ccd_vs_rcd,richtarik2016parallel,Wright2015,nesterov2012efficiency,beck2013convergence}). While these three measures can be related to each other (Jensen's inequality yields $\Ly_1^2 \leq \Ly_2$ and strong convexity enables lower and upper bounding $\Ly_3$ between constant positive multiples of $\Ly_2$), we will provide different analyses for each of the measures to obtain the tightest estimates.

In the above definitions, expectations can be removed for deterministic algorithms such as CCD. By Jensen's inequality, we have that  $\Ly_1^2(x_\text{CD}^k) \le \Ly_2(x_\text{CD}^k)$ for all $k$. For a strongly convex function $f$, $\Ly_3$ can be lower and upper bounded between constant positive multiples of $\Ly_2$.

To study convergence rate of CCD, RCD, and RPCD with respect to improvement sequence $\Ly_1$, we use the operators derived in the previous section that represent one iterate or one epoch. For CCD and RPCD, we have from \eqref{eq:ccd} and \eqref{eq:rpcd} together with \eqref{eq:Brpcd} that 
\begin{equation*}
	\E_\ell x_\text{CD}^{(\ell+1) n} = \cd \, x_\text{CD}^{\ell n},
\end{equation*}
where $\E_\ell$ denotes the expectation with respect to the random variables in epoch $\ell$ given $x_\text{CD}^{\ell n}$. (We have $\cd=\ccd$ for CCD and $\cd=\rpcdE$ for RPCD.) Note that the random variables in each epoch are independent and identically distributed across different epochs for RCD and RPCD. Therefore, by using the law of iterated expectations, we obtain
\begin{equation*}
	\E x_\text{CD}^{(\ell+1) n} = \cd^\ell \, x^0,
\end{equation*}
where $\E$ here denotes the expectation with respect to {\em all} random variables arising in the algorithm.
Hence, the {\em worst-case convergence rate} with respect to $\Ly_1$ can be expressed as
\begin{equation}
	\sup_{x^0\in\R^n} \left( \frac{\norm{\E x_\text{CD}^{\ell n}}}{\norm{x^0}} \right)^{1/\ell} = \sup_{x^0\in\R^n} \left( \frac{\norm{\cd^\ell \, x^0}}{\norm{x^0}} \right)^{1/\ell} = \norm{\cd^\ell}^{1/\ell}. \label{eq:rmk1}
\end{equation}
When $\cd$ is a symmetric matrix (as in RPCD), we have $\norm{\cd^\ell}^{1/\ell} = \rho(\cd)$. Hence, \eqref{eq:rmk1} yields a {\em per-epoch} worst-case convergence rate of $\rho(\rpcdE)$ for RPCD. When $\cd$ is asymmetric (which is the case for CCD), we have by Gelfand's formula $\lim_{\ell\to\infty}\norm{\cd^\ell}^{1/\ell} = \rho(\cd)$. Thus, $\rho(\ccd)$ represents an {\em asymptotic} worst-case convergence rate measure for CCD.

For RCD, a similar derivation involving a single iteration (rather than one epoch) yields from \eqref{eq:rcd-definition} and \eqref{eq:Brcd} that 
\begin{equation*}
\E_k x_\text{RCD}^{k+1} = \rcdE \, x_\text{CCD}^k.
\end{equation*}
Similar reasoning to the above yields a {\em per-iteration} worst-case convergence rate of $\rho(\rcdE)$, or equivalently a per-epoch rate of $\rho(\rcdE)^n$, for RCD. (Note that, because $\rcdE$ is symmetric, we have $\rho(\rcdE) = \norm{\rcdE}$.)

In our analysis of convergence rate of RCD with respect to improvement sequence $\Ly_2$, it follows from \eqref{eq:rcd-definition} that
\begin{align*}
	\E \norm{x_\text{RCD}^{k+1}}^2 & = (x_\text{RCD}^k)^T \E \left[ (\rcd)^T \rcd \right] x_\text{RCD}^k \\
    	& \leq \norm{\E \left[ (\rcd)^T \rcd \right]} \norm{x_\text{RCD}^k}^2.
\end{align*}
For RPCD, we have similarly from \eqref{eq:rpcd} that 
\begin{align*}
	\E \norm{x_\text{RPCD}^{(\ell+1)n}}^2 & = (x_\text{RPCD}^{\ell n})^T \E \left[ (\rpcd)^T \rpcd \right] x_\text{RPCD}^{\ell n} \\
    	& \leq \norm{\E \left[ (\rpcd)^T \rpcd \right]} \norm{x_\text{RPCD}^{\ell n}}^2.
\end{align*}
The matrices $\E \left[ (\rcd)^T \rcd \right]$ and $\E \left[ (\rpcd)^T \rpcd \right]$ are both symmetric. Convergence rates be obtained from $\rho \left( \E \left[ (\rcd)^T \rcd \right] \right)$ and  $\rho \left( \E \left[ (\rpcd)^T \rpcd \right] \right)$ (or equivalently from the norms of these matrices), the first being a per-iteration convergence rate for RCD under criterion $\Ly_2$, and the second being a per-epoch rate for RPCD under the same criterion. Results along these lines appear in Section~\ref{ssec:lyapunov2and3}.

Finally, in our analysis of convergence rate of RCD with respect to $\Ly_3$, iteration \eqref{eq:rcd-definition} yields
\begin{align*}
	\E f(x_\text{RCD}^{k+1}) & = (x_\text{RCD}^k)^T \E_k \left[ (\rcd)^T A \rcd \right] x_\text{RCD}^k \\
    	& = (A^{1/2} x_\text{RCD}^k)^T \E_k \left[ A^{-1/2} (\rcd)^T A \rcd A^{-1/2} \right] A^{1/2} x_\text{RCD}^k \\
    	& \leq \norm{\E_k \left[ A^{-1/2} (\rcd)^T A \rcd A^{-1/2} \right]} \norm{A^{1/2} x_\text{RCD}^k}^2.
\end{align*}
A similar analysis applied to the RPCD update formula \eqref{eq:rpcd} yields
\begin{equation*}
	\E f(x_\text{RPCD}^{(\ell+1)n})  \leq \norm{\E_{\ell} \left[ A^{-1/2} (\rpcd)^T A \rpcd A^{-1/2} \right]} \norm{A^{1/2} x_\text{RPCD}^{\ell n}}^2.
\end{equation*}
We will show that the matrices 
in these two bounds are symmetric. Thus, our convergence rate characterizations for RCD and RPCD with
respect to  $\Ly_3$ (see Section \ref{ssec:lyapunov2and3}) will involve the norms (equivalently, the spectral radii) of these two matrices.

\begin{remark}
Note that for improvement sequence $\Ly_1$, the asymptotic worst-case convergence rate of the algorithm can be simply computed as the spectral radius of the expected iteration matrix. Furthermore, this bound is tight in the sense that there can be no smaller contraction rate $c_1$, for which an inequality of the type $\Ly_1(x_\text{CD}^{\ell n}) \leq c_1^\ell \, \Ly_1(x^0)$ asymptotically holds for all $x^0\in\R^n$. Therefore, in Section \ref{ssec:lyapunov1}, we compare the worst-case convergence rates of CCD, RCD and RPCD with respect to $\Ly_1$ through a tight analysis (in Proposition~\ref{thm:ccd}). We analyze the ratio of the convergence rates of RCD and RPCD in Proposition~\ref{thm-monotonic-rpcd-speedup}. On the other hand, for improvement sequences $\Ly_2$ and $\Ly_3$, we consider per-iteration and per-epoch upper bounds that are not necessarily asymptotically tight. Using a similar argument to \eqref{eq:rmk1}, we can formulate the worst-case contraction factors for $\Ly_2$ and $\Ly_3$, but they would involve computation of powers of matrices (e.g., $\E \left[ (\cdk^\ell)^T \cdk^\ell \right]$ and $\E \left[ A^{-1/2} (\cdk^\ell)^T A \cdk^\ell A^{-1/2} \right]$), which does not admit a closed form characterization. Hence, in Section \ref{ssec:lyapunov2and3}, we compare the convergence rates of RCD and RPCD based on per-iteration and per-epoch improvement rates, as has been done previously in the literature \cite{sun2016worst,wrightRPCD15,wrightRPCD17}.
\end{remark}

\section{Prior work on CD methods with random permutations}\label{sec:prior}
In this section, we survey the known results on the performance of RPCD. There are several recent works that study the effects of random permutations in the convergence behavior of CD methods \cite{oswald2017random,wrightRPCD15,wrightRPCD17,sun2016worst}. To unify the randomization parameters (in RCD and RPCD) and the component-wise Lipschitz constants in different papers, we (without loss of generality) make the following assumption throughout the rest of the paper
\begin{equation}\label{eq:unit-diag}
	A_{ii} = 1, \quad \mbox{for all} \quad i\in\{1,2,\dots,n\}.
\end{equation}
This can always be satisfied by scaling the optimization variable, i.e., by setting $x = D^{-1/2}\tilde{x}$ in \eqref{eq:obj1} and minimizing over $\tilde{x} \in \R^n$ (see e.g. \cite{wrightRPCD17,ccd_vs_rcd}).

Recently, Oswald and Zhou \cite{oswald2017random} analyzed the effects of random permutations for the successive over-relaxation (SOR) method, which is equivalent to the CD method with exact line search for a particular choice of algorithm parameter. They consider quadratic problems whose Hessian matrix is positive semidefinite and present convergence guarantees for SOR iterations with random permutations, which implies the following guarantee on the performance of RPCD.

\begin{theorem}\label{theo-oswald}\cite[Theorem 4]{oswald2017random}
Let $f$ be a quadratic function of the form \eqref{eq:obj1}, where the Hessian matrix $A$ has unit diagonals. Then, for any solution $x^*$, the RPCD algorithm enjoys the following guarantee
\begin{equation}\label{eq:oswald_rpcd}
	\E f(\xp^{\ell n}) - f(x^*) \leq  \left (1 - \frac{\mu}{(1+L)^2}\right)^{\ell} \left(f(x^0) - f(x^*)\right).
\end{equation}
\end{theorem}

Theorem \ref{theo-oswald} provides a convergence rate guarantee on the performance of RPCD for general quadratic functions. Under the same assumptions in Theorem \ref{theo-oswald}, the best known upper bound on the performance of RCD is given by \cite[Theorem~5]{nesterov2012efficiency}:
\begin{equation}
	\E \left[ \frac{1}{2} \norm{\xr^k-\xs}^2 + f(\xr^k) - f(x^*) \right] \leq\left( 1- \frac{2 \mu}{n(1+\mu)} \right)^k \left( \frac{1}{2} \norm{x^0-\xs}^2 + f(x^0) - f(x^*) \right). \label{eq:rcd.rate.R}
\end{equation}
This shows that the the upper bound on the performance of RCD per-epoch is approximately $\left( 1-\frac{2\mu}{n(1+\mu)} \right)^n \approx 1-\frac{2\mu}{1+\mu}$, whereas it follows from \eqref{eq:oswald_rpcd} that the upper bound on the performance of RPCD can be as large as $1-\frac{\mu}{(1+n)^2}$ since $L \leq \tr{A} = n$. These bounds suggest that RPCD may require $\bigO(n^2)$ times more iterations than RCD to guarantee an $\epsilon$-optimal solution. However, empirical results show that RPCD often outperforms RCD in machine learning applications \cite{Recht2012,Bottou2009}. Furthermore, it has been conjectured that the expected performance of RPCD should be no worse than the expected performance of RCD \cite{Recht2012} (see also \cite{Ward16AMGM,Zhang14AMGM} for related work on this conjecture). This motivates to derive tight bounds for the convergence rate of RPCD and compare them with the known bounds on the convergence rate of RCD.

A similar phenomenon has been observed for CCD in comparison to RCD. In particular, the tightest known convergence rate results on the performance of CCD (see \cite{beck2013convergence,sun2016worst,sun2015improved}) suggest that CCD may require $\bigOt(n^2)$ times more iterations than RCD to guarantee an $\epsilon$-optimal solution. To understand this gap in the convergence rate bounds, Sun and Ye \cite{sun2016worst} focused on the quadratic problem in \eqref{eq:obj1} with the following permutation invariant\footnote{$A$ is a permutation invariant matrix if $PAP^T=A$, for any permutation matrix $P$.} Hessian matrix
\begin{equation} \label{eq:Ainvariant}
	A = \ddd I + (1-\ddd) \bfone \bfone^T, \quad \mbox{where} \quad \delta \in (0,n/(n-1)).
\end{equation}
In particular, the authors considered a worst-case initialization and the case when $\delta$ is close to $0$, for which $L=\bigO(n)$.\footnote{Since $A$ has two eigenvalues: $\delta+n(1-\delta)$ with multiplicity $1$ and $\delta$ with multiplicity $n-1$, the Lipschitz constant becomes $L=\delta+n(1-\delta)$, for $\delta\leq1$; and as $\delta\to0$, $L \to n$.} For this problem, they showed that CCD with the worst-case initialization indeed requires $\bigO(n^2)$ times more iterations than RCD to return an $\epsilon$-optimal solution. They also provided rate comparisons between RPCD and CCD without providing a comparison between RPCD and RCD, which is presented in the following theorem.

\begin{theorem}\label{prop: slower, iterates} \cite[Proposition~3.4]{sun2016worst}
Let $K_{\mathrm{CCD}}(\epsilon)$, $K_{\mathrm{RCD}}(\epsilon) $ and $K_{\mathrm{RPCD}}(\epsilon)$ be the minimum number of epochs for CCD, RCD and RPCD (respectively) to achieve (expected) relative error
$$
\frac{  \| \E(x_{\text{CD}}^k) - x^* \| }{ \|x^0 - x^* \| } \leq \epsilon,
$$
for initial point $x^0\in\R^n$ (for CCD, the expectation operator can be ignored). There exists a quadratic problem, whose Hessian matrix $A$ satisfies \eqref{eq:Ainvariant} for some $\delta$ around zero, such that
\begin{subequations}\label{ite, compare CD with others}
\begin{align}
      \frac{ K_{\mathrm{CCD} }(\epsilon ) }{ K_{\mathrm{RCD} }(\epsilon ) }    &   \geq \frac{n^2 }{ 2 \pi^2 } \approx \frac{n^2}{20}, \label{compare CD with RCD, coro} \\
      \frac{ K_{\mathrm{CCD} }(\epsilon ) }{ K_{\mathrm{RPCD} }(\epsilon ) }   &   \geq \frac{n (n+1) }{ 2 \pi^2 } \approx \frac{n(n+2)}{20}.
 \end{align}
  \end{subequations}
\end{theorem}

Theorem \ref{prop: slower, iterates} shows that the worst-case performance (in improvement sequence $\Ly_1$) of RPCD and RCD is $\bigO(n^2)$ times faster than that of CCD. In a follow-up work, Lee and Wright \cite{wrightRPCD15} considered the same problem as \cite{sun2016worst} (see \eqref{eq:Ainvariant}) for the small $\delta$ case and presented asymptotic and non-asymptotic analyses of RPCD with respect to improvement sequence $\Ly_3$, presented in the following theorem.

\begin{theorem}\label{thm-rpcd-ima}\cite[Theorem 3.3]{wrightRPCD15}
Consider the quadratic problem \eqref{eq:obj1} with the Hessian matrix $A$ given by \eqref{eq:Ainvariant}, where $\delta\in(0, 0.4)$ and $n \geq 10$. For any $x^0\in\R^n$, RPCD has the following non-asymptotic convergence guarantee 
\begin{equation}
	\E f(\xp^{\ell n}) - f(\xs) \leq (1-2\delta+4\delta^2)^\ell R_0,
\end{equation}
where $R_0$ is a constant depending on $x_0$ and $\delta$. Furthermore, RPCD iterates enjoy an asymptotic convergence rate of
\begin{equation}
	\lim_{\ell\to\infty} \left( \E f(\xp^{\ell n}) - f(\xs) \right)^{1/\ell}  = 1-2\ddd - \frac{2\ddd}{n} + 2\ddd^2 + \bigO\left(\frac{\ddd^2}{n}\right) + \bigO(\ddd^3).
\end{equation}
\end{theorem}

Theorem \ref{thm-rpcd-ima} shows that for the particular class of quadratic problems whose Hessian matrix satisfies \eqref{eq:Ainvariant}, the convergence rate (in improvement sequence $\Ly_3$) of RPCD is faster than that of RCD in \eqref{eq:rcd.rate.R} in terms of the best known upper bounds. This is the first theoretical evidence that supports the empirical results showing RPCD often outperforms RCD \cite{Recht2012}. In a follow-up work \cite{wrightRPCD17}, Lee and Wright generalize the results of Theorem \ref{thm-rpcd-ima} to quadratic problems, whose Hessian matrix satisfies
\begin{equation} \label{eq:Ainvariant-generalized}
	A = \ddd I + (1-\ddd) u u^T, \quad \mbox{where} \quad \ddd \in (0,n/(n-1)),
\end{equation}
where $u\in\R^n$ is a vector with elements of size $\bigO(1)$ (this generalizes \eqref{eq:Ainvariant} that corresponds to $u=\bfone$). The conclusions are similar to \cite{wrightRPCD15}, but the analysis is different because $A$ is no longer a permutation-invariant matrix.

\section{Performance of RPCD vs RCD on a class of diagonally dominant matrices}\label{sec:4}
As described in the previous section, the existing works \cite{sun2016worst,wrightRPCD15} analyze the performance of RPCD for quadratic problems, whose Hessian satisfies \eqref{eq:Ainvariant} for small $\delta$. Here, we consider the other extreme, i.e., the $\delta>1$ case, and provide tight convergence rate comparisons between RPCD, RCD and CCD with respect to all there improvement sequences defined in Section \ref{ssec:rate_criteria}. In deriving convergence rate guarantees, we do not resort to the tools that are used in the earlier works on RPCD \cite{sun2016worst,wrightRPCD15,wrightRPCD17}. Instead, we present a novel analysis based on Perron-Frobenius theory that enables us to compute convergence rate bounds for all three criteria. For notational simplicity, we introduce the reformulation $\alpha=\delta-1$, which yields
\begin{equation}\label{def-An}
	A = (1+\alpha)I - \alpha \ones \ones^T, \quad \mbox{where} \quad \alpha \in (0,1/(n-1)).
\end{equation}
It is simple to check that $A$ has one eigenvalue at $1-(n-1)\alpha$ with the corresponding eigenvector $\ones$ and other $n-1$ eigenvalues equal to $1+\alpha$. In particular, as $\alpha$ goes to zero, the condition number of $A$ gets smaller and in the limit $A$ is the identity matrix. On the other hand, as $\alpha\to \frac{1}{n-1}$, the matrix gets ill-conditioned. Therefore, the parameter
    \beq t:=\max_i \frac{\sum_{j\neq i}A_{ij}}{A_{ii}} = \alpha(n-1) \in (0,1) \label{eq:t-definition}
    \eeq
is a measure of diagonal dominance. In the remainder of this section, we analyze the performance of RPCD, RCD and CCD in improvement sequence $\Ly_1$ and the performance of RPCD and RCD in improvement sequences $\Ly_2$ and $\Ly_3$ with respect to this diagonal dominance measure.

\subsection{Convergence rates of RPCD, RCD and CCD in improvement sequence $\Ly_1$}\label{ssec:lyapunov1}
In this section, we compare convergence rates of RPCD, RCD and CCD, where improvement sequence $\Ly_1(x^k) = \norm{\E x^k -\xs}$ is chosen as the convergence criterion (as in Theorem \ref{prop: slower, iterates}). As we highlighted in Section \ref{ssec:rate_criteria}, we first compute the expected iteration matrices of the RPCD and RCD algorithms, and show that they are symmetric. Then, we compute their spectral radii to conclude the per-epoch worst-case convergence rate of RPCD and RCD, and analyze their ratio in Proposition \ref{thm-monotonic-rpcd-speedup}. We also show that the asymptotic worst-case convergence rate of CCD is faster than that of RPCD and RCD in Proposition \ref{thm:ccd}.

We begin our discussion by writing the expected RPCD iterates (see \eqref{eq:rpcd} and \eqref{eq:Brpcd}) as follows
\begin{equation}\label{eq:CCDEqualsCCDpi}
	\E_\ell \xp^{(\ell+1)n} = \rpcdE \, \xp^{\ell n}.
\end{equation}
Note that since the Hessian matrix $A$ is permutation invariant, the iteration matrix of the CCD-$\pi$ algorithm for any cyclic order $\pi$ is equal to the iteration matrix of the standart CCD algorithm, i.e., $\ccd=\ccdpi$ for all orders $\pi$. Therefore, we have $\rpcdE = \E_\pi [P_\pi \ccd P_\pi^T] = \E_P [P \ccd P^T]$, where we drop the subscript $\pi$ from the matrices for notational simplicity. In order to obtain a formula for $\rpcdE$, we first reformulate the CCD iteration matrix in \eqref{eq:ccd} as follows
\begin{equation*}
	\ccd = (I-N)^{-1}N^T = I - (I-N)^{-1}(I-N-N^T) =   I - \Gamma^{-1} A,
\end{equation*}
where $\Gamma = I-N$. Using this reformulation, the expected iteration matrix of RPCD can computed as follows
\begin{equation*}
	\rpcdE = \E_P \left[ P \ccd P^T \right] =  \E_P \left[ P(I - \Gamma^{-1} A)P^T \right] = I - \E_P \left[ P\Gamma^{-1}P^T \right] A,
\end{equation*}
where we used the fact that $PP^T = I$ and $AP^T = P^TA$. For the case the Hessian matrix $A$ satisfies \eqref{def-An}, $\Gamma^{-1}$ can be explicitly computed as
\begin{eqnarray} \label{eqn-gamma-inv}
	\Gamma^{-1} = \mbox{toeplitz} (c,r),
\end{eqnarray}
where $\mbox{toeplitz} (c,r)$ denotes the Toeplitz matrix with the first column $c$ and the first row $r$, which are given by
\begin{eqnarray*}		
		c = \begin{bmatrix} 1, & \alpha, & \alpha(1+\alpha), & \alpha(1+\alpha)^2, & \dots, & \alpha(1+\alpha)^{n-2}  \end{bmatrix}^T,  \quad
		r = [1, 0 , 0, \dots, 0].
\end{eqnarray*}
In order to compute $\E_P \left[ P\Gamma^{-1}P^T \right]$, we use the following lemma, which states that expectation over all permutations separately averages the diagonal and off-diagonal entries of the permuted matrix.

\begin{lemma}\cite[Lemma 3.1]{wrightRPCD15}\label{lem:wright_lemma}
Given any matrix $Q \in \R^{n\times n}$ and permutation matrix $P$ selected uniformly at random from the set of all permutations, we have
\begin{equation*}
	\E_P[P Q P^T] = \tau_1 I + \tau_2 \bfone\bfone^T,
\end{equation*}
where
\begin{equation}\label{eq:wright-lemma}
\tau_2 = \frac{\bfone^T Q \bfone - \trace(Q)}{n(n-1)} \quad \mbox{and} \quad \tau_1 = \frac{\trace (Q)}{n} - \tau_2.
\end{equation}
\end{lemma}

Letting $Q=\Gamma^{-1}$ in Lemma \ref{lem:wright_lemma}, we observe that the matrix $\E_P [P\Gamma^{-1}P^T]$ has diagonals equal to one and all the off-diagonal entries equal to each other:
\begin{equation} \label{eq:expected-gamma-inverse-matrix}
	\E_P [P\Gamma^{-1}P^T] = (1-\gamma) I + \gamma \bfone\bfone^T,
\end{equation}
where $\gamma$ can be found as the average of the off-diagonal entries of $\Gamma^{-1}$. The following lemma (whose proof is given in Appendix \ref{app:lemm-gamma-proof}) provides an explicit expression for $\gamma$.

\begin{lemma}\label{lemm-gamma}
For any $\alpha \in (0,1/(n-1))$, we have $$\gamma = \frac{(1+\alpha)^n - \alpha n - 1}{\alpha n(n-1)}, $$
where $\gamma$ denotes the off-diagonal entries of $\E_P [P\Gamma^{-1}P^T]$ in \eqref{eq:expected-gamma-inverse-matrix}.
\end{lemma}

Using Lemma \ref{lemm-gamma}, it follows from the definition of $A$ in \eqref{def-An} and equation \eqref{eq:expected-gamma-inverse-matrix} that
\begin{equation*}
	\rpcdE = I - \E_P [P\Gamma^{-1}P^T] A = ((n-1)\gamma-\beta) I + \beta \bfone\bfone^T,
\end{equation*}
where 
\beq
	\beta = \alpha - \gamma + \alpha \gamma (n-2). \nonumber
\eeq
Since $\rpcdE$ is a symmetric matrix, then by \eqref{eq:rmk1}, it suffices to compute the spectral radius of $\rpcdE$ to obtain the worst-case performance of RPCD with respect to improvement sequence $\Ly_1$. To this end, we note that for any $\alpha \in (0,1/(n-1))$, $\rpcdE >0$ since $\rpcdE = \E_P [P \ccd P^T]$ and $\ccd\geq0$ with at least one strictly positive entry in both the diagonal and off-diagonal parts (see also \eqref{eq-ccd iter matrix for ex} for an explicit formula of $\ccd$). Then, by the Perron-Frobenius Theorem \cite[Lemma 2.8]{varga2009matrix}, we have
\begin{align*}
	\rho(\rpcdE) & = \sum_{j=1}^n [\rpcdE]_{ij}, \quad \mbox{for all $i\in[n]$} \\
    	& = (n-1) (\gamma \alpha + \beta) \\
    	& = (n-1) (\alpha - \gamma + \alpha \gamma (n-1)) \\
        & = 1 - \left[\left(1-\alpha(n-1)\right) \left(1+\gamma(n-1)\right)\right].
\end{align*}
Substituting the formula for $\gamma$ from Lemma \ref{lemm-gamma} above, we obtain the spectral radius of the RPCD iteration matrix as follows
\begin{equation} \label{eq-rho-RPCD}
	\rho(\rpcdE) = 1 - \left(1-\alpha(n-1)\right) \frac{(1+\alpha)^n  - 1}{\alpha n} = 1 - \frac{1-t}{n} \left( \frac{\left(1+\frac{t}{n-1}\right)^n  - 1}{\frac{t}{n-1}} \right),
\end{equation}
where $t=\alpha(n-1)$ denotes the diagonal dominance factor (as defined in \eqref{eq:t-definition}).

For the RCD algorithm, on the other hand, we have (by \eqref{eq:rcd-definition} and \eqref{eq:Brcd}) the following expected iterates
\begin{equation*}
	\E_k \xr^{k+1} = \rcdE \, \xr^k, \quad \mbox{where} \quad \rcdE = I-\frac{1}{n}A.
\end{equation*}
Since $A$ is a symmetric matrix, then by \eqref{eq:rmk1}, the per-epoch worst-case asymptotic rate of RCD with respect to improvement sequence $\Ly_1$ can be found as
\begin{equation*}
	\rho(\rcdE)^n = \left(1-\frac{1}{n}\lambda_{\min}(A)\right)^n =  \left( 1-\frac{1-t}{n} \right)^n.
\end{equation*}
In Proposition \ref{thm-monotonic-rpcd-speedup}, we compare the performance of RPCD and RCD with respect to improvement sequence $\Ly_1$. To this end, we define
\begin{equation}
	s(t,n) = \frac{-\log \rho (\rpcdE)}{-\log \rho(\rcdE)^n}, \label{eq:s-definition}
\end{equation}
(where $\log$ denotes the natural logarithm), which is equal to the ratio between the number of epochs required to guarantee $\norm{\E x^{\ell n}-\xs}\leq\epsilon$ for RCD and RPCD algorithms. In particular $s(t,n)>1$ implies RPCD has a faster worst-case convergence rate than RCD. In the following theorem, we show that RPCD is faster than RCD for any $t\in(0,1)$ and $n\geq2$, and quantify the rate of improvement.

\begin{proposition}\label{thm-monotonic-rpcd-speedup}
The following statements are true: 
\begin{itemize}
	\item [$(i)$] The function $s(t,n)$ is strictly decreasing in $t$ over $(0,1)$.
	\item [$(ii)$] $\lim_{t\to 0}s(t,n) = \infty.$
	\item [$(iii)$] Let $g(n):= \lim_{t\to 1} s(t,n)$. We have $g(n) \in [3/2,e-1)$, for any $n\geq 2$. Furthermore, $g(n)$ is strictly increasing in $n\geq2$ satisfying
\begin{equation*}
    g(2)=3/2 \quad \mbox{and} \quad \lim_{n\to\infty} g(n) = e-1.
\end{equation*}
\end{itemize}	
\end{proposition}

\begin{figure}[!tbp]
  \centering
  \begin{minipage}[b]{0.49\textwidth}
    \includegraphics[width=\textwidth]{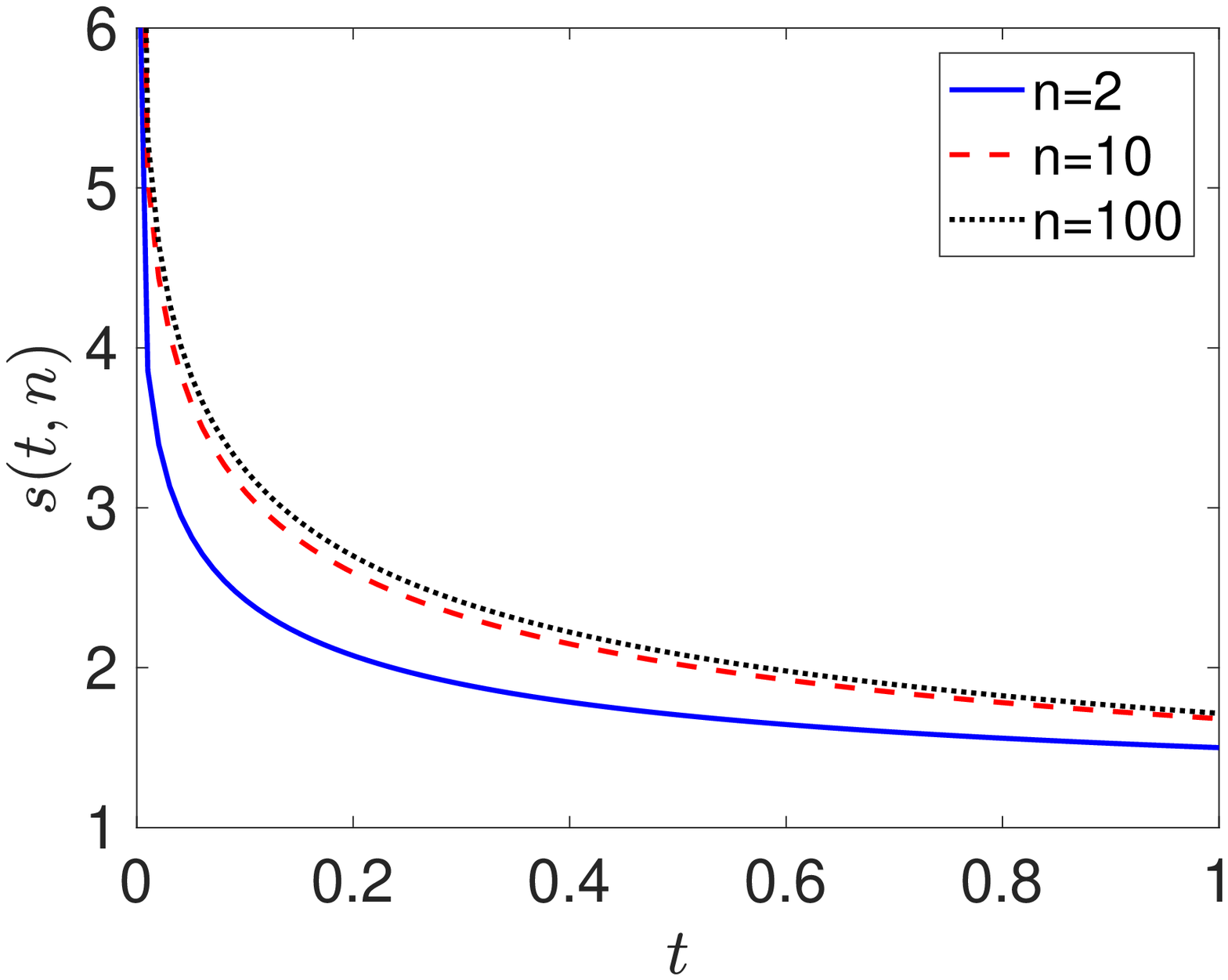}
  \end{minipage}
  \hfill
  \begin{minipage}[b]{0.49\textwidth}
    \includegraphics[width=\textwidth]{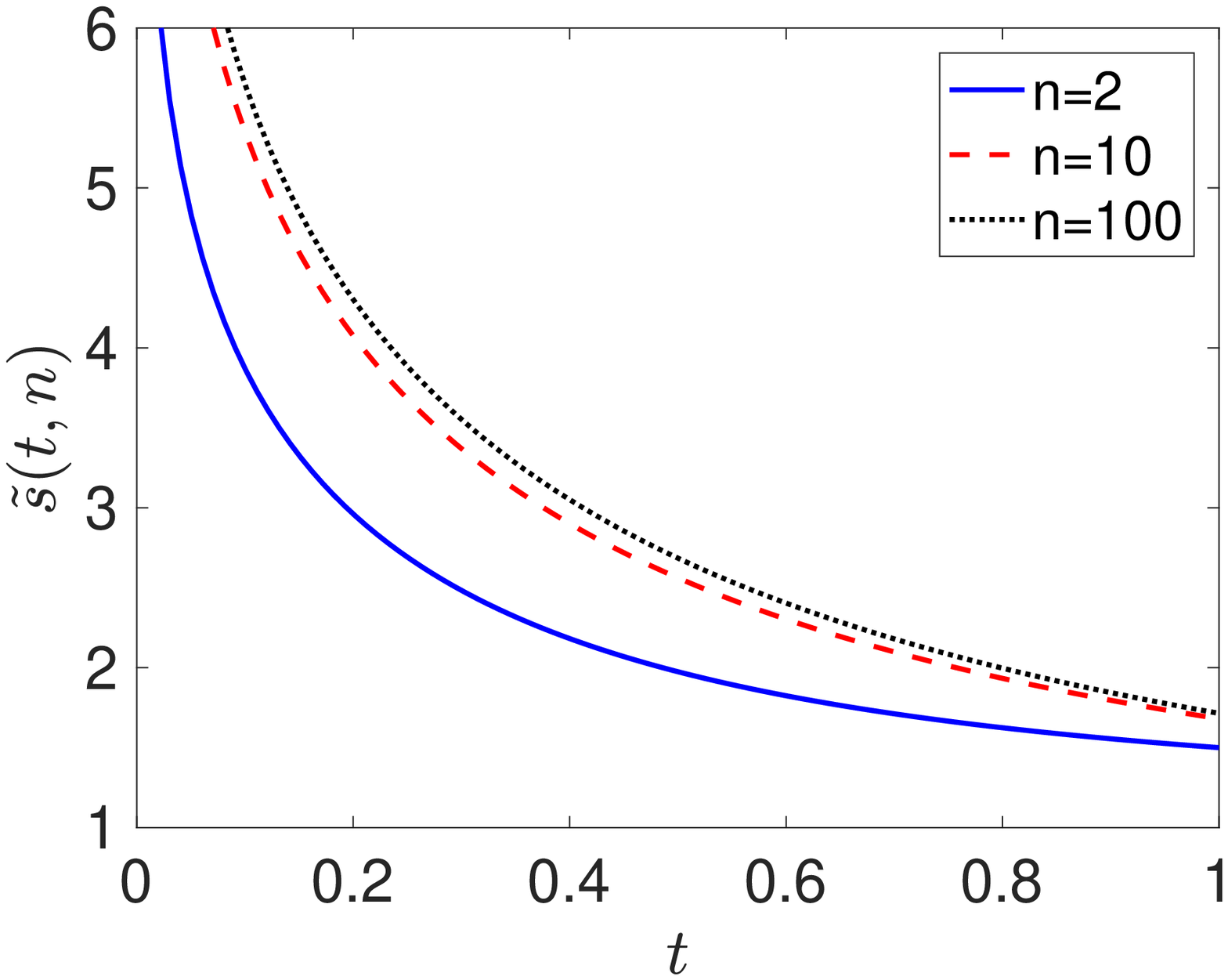}
  \end{minipage}
  \caption{Plot of $s(t,n)$ and $\tilde{s}(t,n)$ versus $t\in(0,1)$ for different values of $n$.}
  \label{fig:s}
\end{figure}

A consequence of Proposition \ref{thm-monotonic-rpcd-speedup} is that RPCD is faster than RCD in the worst-case, for every $t \in (0,1)$ by a factor $s(t,n) > 1$. Furthermore, the amount of acceleration $s(t,n)$ goes to infinity as $\alpha\to0$ for any $n$ fixed. This shows that as the matrix $A$ becomes more and more well-conditioned (as $\alpha \to 0$), the amount of speed-up $s(t,n)$ we obtain with RPCD with respect to RCD goes to infinity. This is consistent with the observation that cyclic orders work well for diagonal-like matrices that are well-conditioned (see e.g. \cite{varga2009matrix}). Proposition \ref{thm-monotonic-rpcd-speedup} is illustrated in Figure \ref{fig:s} (left panel), where we plot the parameter $s(t,n)$ as a function of $t$ for different values of $n$.

We next compare the convergence rate of CCD with respect to RPCD and RCD. To this end, as we discuss in Section \ref{ssec:rate_criteria} (cf. \eqref{eq:rmk1}), we use $\rho(\ccd)$ as the asymptotic per epoch worst-case convergence rate of CCD, whereas for comparison to RCD, we use a per-epoch rate of $\rho(\rcdE)^n$. Note that as discussed in \eqref{eq:CCDEqualsCCDpi}, $\ccd=\ccdpi$ for all $\pi$, and hence $\rho(\ccd)=\rho(\ccdpi)$ for all $\pi$. Although, explicit calculation of $\rho(\ccd)$ appears to be challenging, we prove that the known upper bounds \cite[Theorem 4.12]{ccd_vs_rcd} on $\rho(\ccd)$ is tighter than $\rho(\rpcdE)$, which together with Proposition \ref{thm-monotonic-rpcd-speedup} imply the following result.

\begin{proposition}\label{thm:ccd}
Let $f$ be a quadratic function of the form \eqref{eq:obj1}, whose Hessian matrix given by \eqref{def-An}. Then, the expected iteration matrices of CCD, RPCD and RCD satisfy
\begin{equation}
	\rho(\ccd) < \rho(\rpcdE) < \rho(\rcdE)^n,
\end{equation}
for any $\alpha\in(0,1/(n-1))$ and $n\geq2$.
\end{proposition}

\subsection{Convergence rates of RPCD and RCD in improvement sequences $\Ly_2$ \& $\Ly_3$} \label{ssec:lyapunov2and3}
In this section, we compare the rate of RPCD and RCD with respect to improvement sequences $\Ly_2$ and $\Ly_3$. When the Hessian matrix $A$ satisfies \eqref{def-An}, the smallest eigenvalue of $A$ can be found as follows
\begin{equation} \label{eq:mu-definition}
	\mu = 1-t = 1-\alpha (n-1).
\end{equation}
Plugging this value in the convergence guarantee of RCD in \eqref{eq:rcd.rate.R}, we can obtain a convergence guarantee on both improvement sequences $\Ly_2$ and $\Ly_3$ as the left hand-side of \eqref{eq:rcd.rate.R} upper bounds both $2\Ly_2$ and $\Ly_3$. However, for the particular problem class we consider in this paper, we derive a tighter convergence rate guarantee for RCD in the next proposition, whose proof is deferred to Appendix \ref{app:rcd-bound}.

\begin{proposition}\label{lemma-rcd-upper bound}
Let $f$ be a quadratic function of the form \eqref{eq:obj1}, whose Hessian matrix given by \eqref{def-An}. Then, RCD iterations satisfy
\begin{equation}
	\E \|\xr^{k} - x^*\|^2 \leq \left(1-\frac{2\mu}{n} + \frac{\mu^2}{n} \right)^{k} \|x^0 - x^*\|^2, \label{eq:upper bound-rcd norm square}
\end{equation}
and
\begin{equation}
    \E \left( f(\xr^{k}) - f(x^*) \right) \leq \left(1-\frac{\mu}{n} \right)^{k} \left( f(x^0) - f(x^*) \right). \label{eq:upper bound-rcd subopt}
\end{equation}
\end{proposition}

\begin{figure}[!tbp]
  \centering
  \begin{minipage}[b]{0.49\textwidth}
    \includegraphics[width=\textwidth]{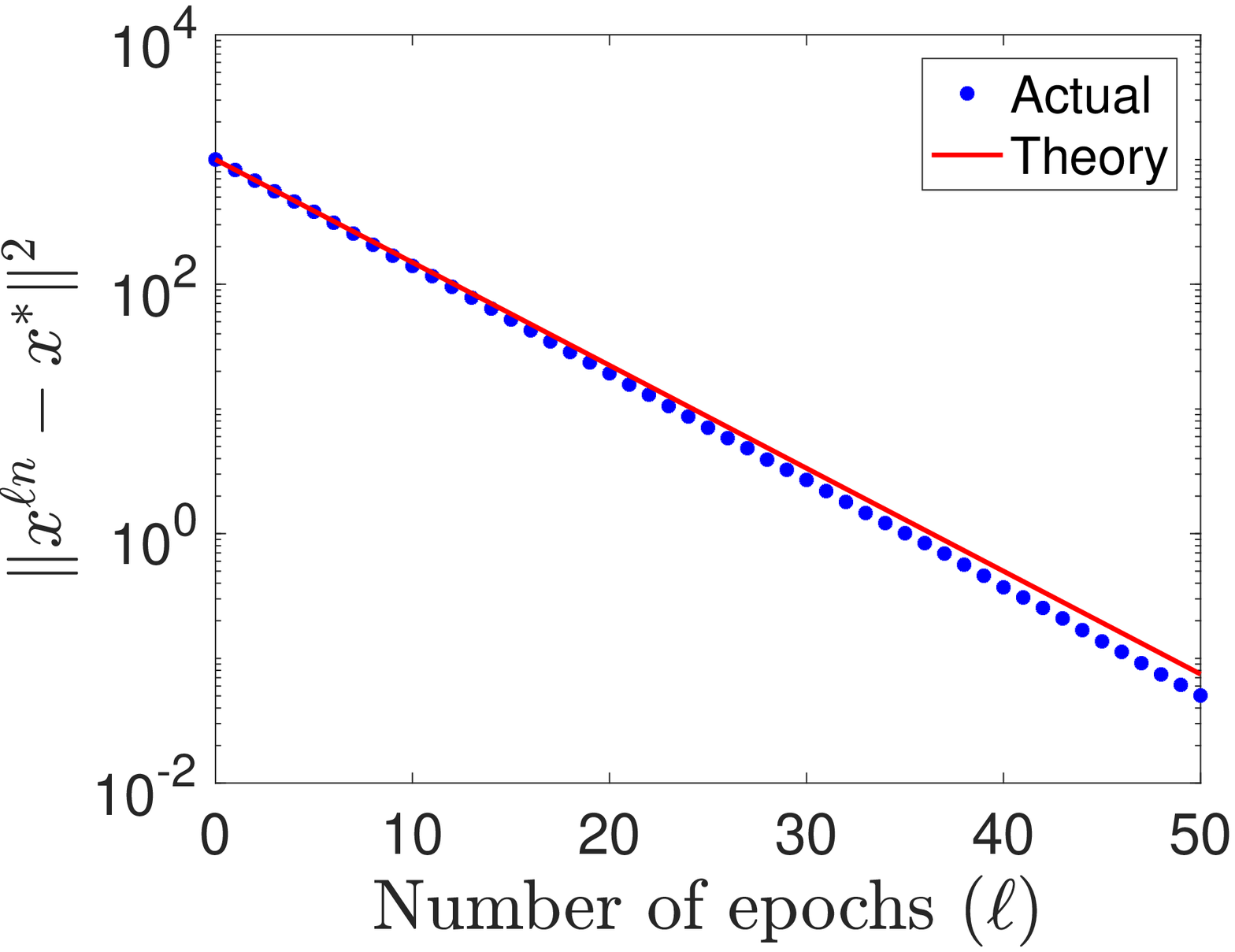}
  \end{minipage}
  \hfill
  \begin{minipage}[b]{0.49\textwidth}
    \includegraphics[width=\textwidth]{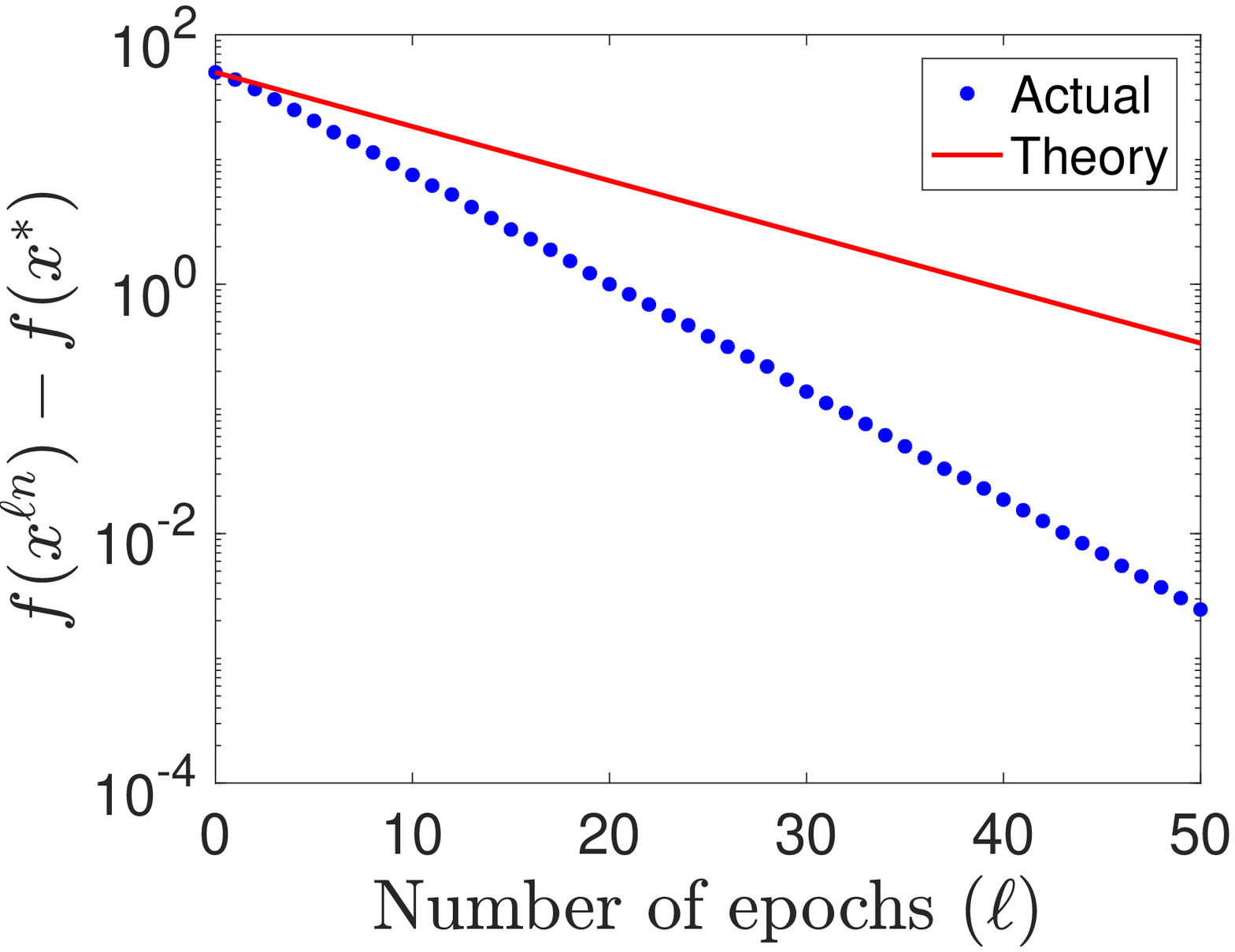}
  \end{minipage}
  \caption{Tightness of the bounds in Proposition \ref{lemma-rcd-upper bound} when $n=1000$ and $\alpha=\frac{0.9}{n-1}$: Left figure for \eqref{eq:upper bound-rcd norm square} and right figure for \eqref{eq:upper bound-rcd subopt}.}
\end{figure}

\begin{remark}
We observe that the upper bound in \eqref{eq:upper bound-rcd norm square} is smaller (tighter) than the upper bound in \eqref{eq:rcd.rate.R} for any $\alpha \in (0,1/(n-1))$ because
	$$  1-\frac{2\mu}{n} + \frac{\mu^2}{n} < 1-\frac{2\mu}{n} + \frac{2\mu^2}{n} = 1-\frac{2\mu(1-\mu)}{n} = 1-\frac{2\mu(1-\mu^2)}{n(1+\mu)} < 1 - \frac{2\mu}{n(1+\mu)},$$
where the inequalities are due to the fact that $\mu = 1-\alpha (n-1) \in (0,1)$.
\end{remark}

We next analyze the performance of RPCD in the following proposition and show that the convergence rate guarantee of RPCD is tighter than the convergence rate guarantee of RCD in Proposition \ref{lemma-rcd-upper bound}. The proof of Proposition \ref{theo-subopt-rpcd} is given in Appendix \ref{app:pot6}.

\begin{proposition}\label{theo-subopt-rpcd}
Let $f$ be a quadratic function of the form \eqref{eq:obj1}, whose Hessian matrix given by \eqref{def-An}. Then, RPCD iterations satisfy
\begin{equation} \label{eq:upper bound-rpcd norm square}
	\E \|\xp^{\ell n} - x^*\|^2 \leq \left( 1 - \frac{2\mu}{n} \left( \frac{(1+\alpha)^n - 1}{\alpha} \right) +  \frac{\mu^2}{n} \left( \frac{(1+\alpha)^{2n} - 1}{\alpha(\alpha+2)}\right) \right)^\ell \, \|x^0 - x^*\|^2,
\end{equation}
and
\begin{equation} \label{eq:rpcd-f-upper-bound}
	\E f(\xp^{\ell n}) - f(x^*) \leq \left( 1- \frac{\mu}{n} \left( \frac{(1+\alpha)^{2n}-1}{\alpha(\alpha+2)} \right) \right)^\ell \left( f(x^0) - f(x^*) \right).
\end{equation}
\end{proposition}

\begin{figure}[!tbp]
  \centering
  \begin{minipage}[b]{0.49\textwidth}
    \includegraphics[width=\textwidth]{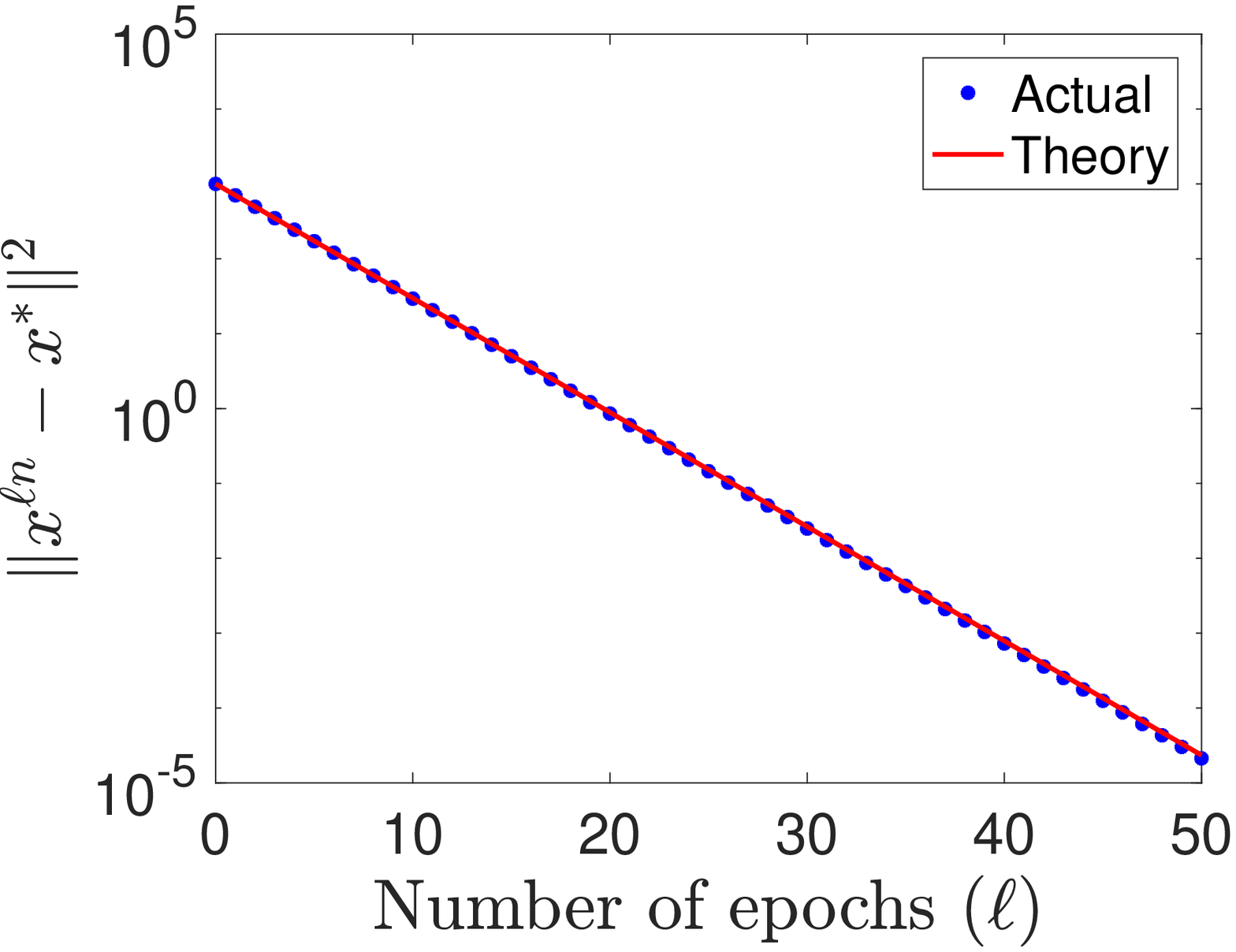}
  \end{minipage}
  \hfill
  \begin{minipage}[b]{0.49\textwidth}
    \includegraphics[width=\textwidth]{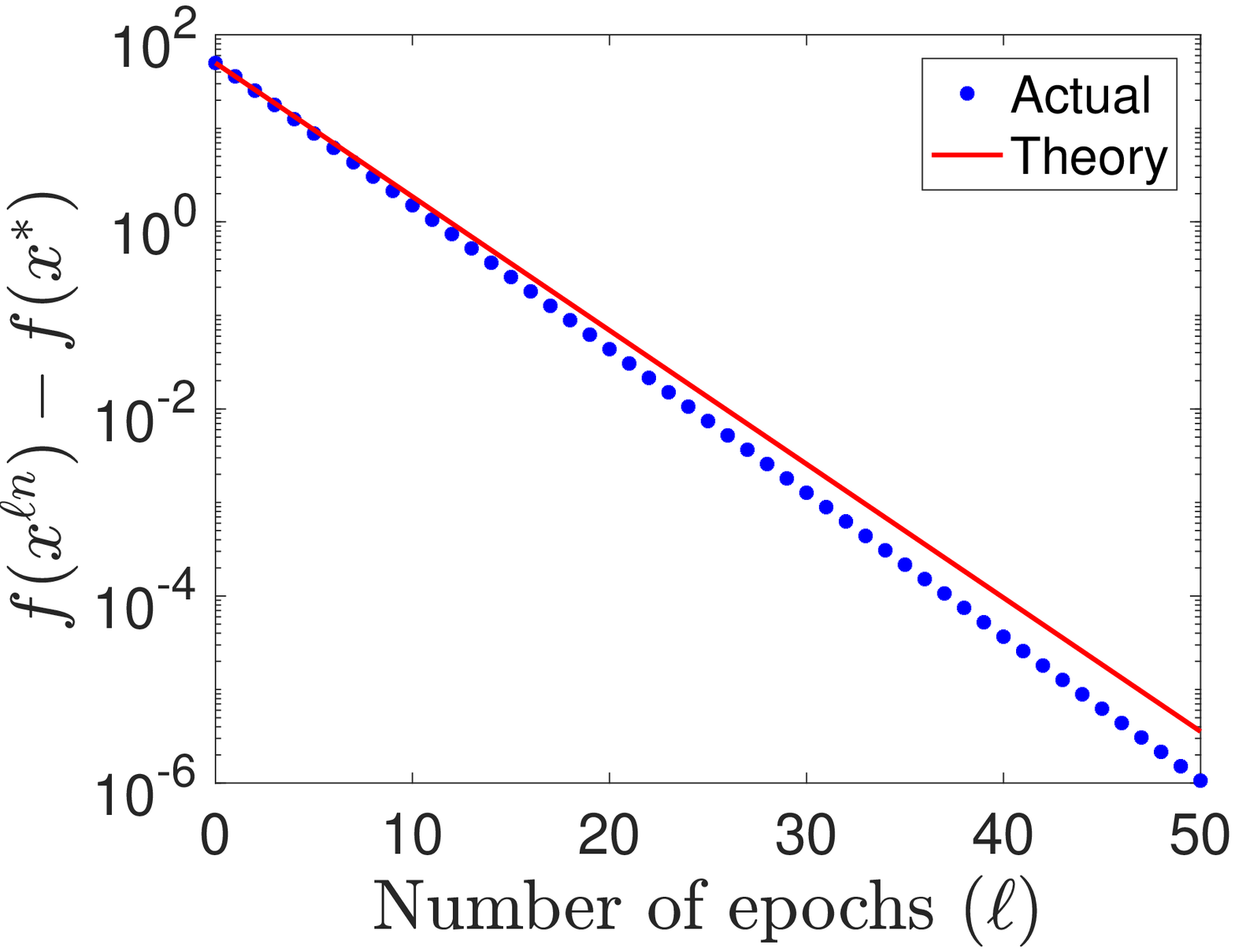}
  \end{minipage}
  \caption{Tightness of the bounds in Proposition \ref{theo-subopt-rpcd} when $n=1000$ and $\alpha=\frac{0.9}{n-1}$: Left figure for \eqref{eq:upper bound-rpcd norm square} and right figure for \eqref{eq:rpcd-f-upper-bound}.}
\end{figure}

We next compare the convergence rates we derive for the RCD and RPCD algorithms. In particular, we consider the convergence rate of both algorithms in improvement sequence $\Ly_2$ since we obtain tighter upper bounds for it. Comparing the convergence rate bounds for RCD and RPCD in \eqref{eq:upper bound-rcd norm square} and \eqref{eq:upper bound-rpcd norm square}, respectively, we can observe that RPCD is faster (in terms of the best known rate guarantees) than RCD by a factor of
\begin{equation*}
	\tilde{s}(t,n) := \frac{-\log\left( 1 - \frac{2\mu}{n} \left( \frac{(1+\alpha)^n - 1}{\alpha} \right) +  \frac{\mu^2}{n} \left( \frac{(1+\alpha)^{2n} - 1}{\alpha(\alpha+2)}\right) \right)}{-n\log\left( 1- \frac{2\mu}{n} + \frac{\mu^2}{n} \right)},
\end{equation*}
which is plotted in Figure \ref{fig:s} (right panel) in the interval $t\in(0,1)$ for different values of $n$. We observe from this figure that the convergence rate bound for RPCD is better than than the one for RCD for all $t\in(0,1)$ and $n\geq2$. Furthermore, the difference in convergence rate bounds increases as $t$ gets smaller, i.e., as the Hessian matrix becomes more diagonally dominant. We can also show that $\tilde{s}(t,n)$ behaves similar to $s(t,n)$ as $t\to 1$, where the limiting values can be found in Proposition \ref{thm-monotonic-rpcd-speedup}.

\section{Numerical Experiments}\label{sec:experiments}
Here we compare the performance of CCD, RPCD, and RCD for the quadratic problem \eqref{eq:obj1} with Hessian matrix \eqref{def-An}. In Figure \ref{fig:worst}, we use a worst-case initialization $x^0=\bfone$, for $n\in\{1000,10000\}$ and $\alpha \in\left\{ \frac{0.01}{n-1},\frac{0.50}{n-1},\frac{0.99}{n-1} \right\}$. We observe that CCD is the faster than RPCD, which is faster than RCD. This behavior is in accordance with the theoretical results in Propositions \ref{thm:ccd}-\ref{theo-subopt-rpcd}. Furthermore, as $\alpha$ decreases, we can see that the ratio between the convergence rates of RPCD and RCD increases, consistent with Proposition~\ref{thm-monotonic-rpcd-speedup} (see also Figure \ref{fig:s}). We can also observe from the right column in Figure \ref{fig:worst} that when $\alpha$ is close to $1/(n-1)$, the ratio between the convergence rates of RPCD and RCD is close to the theoretical limits obtained in Proposition~\ref{thm-monotonic-rpcd-speedup} (see part {\em (iii)}, which shows that the ratio is in the interval $[3/2,\,e-1)$). Figure \ref{fig:random} plots similar results to Figure~\ref{fig:worst}, but for a random initialization rather than worst-case initialization.  Convergence rates depicted in Figure~\ref{fig:random} are similar to those of Figure~\ref{fig:worst}, due to the fact that $x^{\ell n}$ becomes colinear with the vector of ones as $\ell$ increases (as $\bfone$ is the leading eigenvector of the expected iteration matrix), so that the worst-case convergence rate dictates the performance of the algorithms.

\begin{figure}[!tbp]
  \centering
  \begin{minipage}[b]{0.32\textwidth}
    \includegraphics[width=\textwidth]{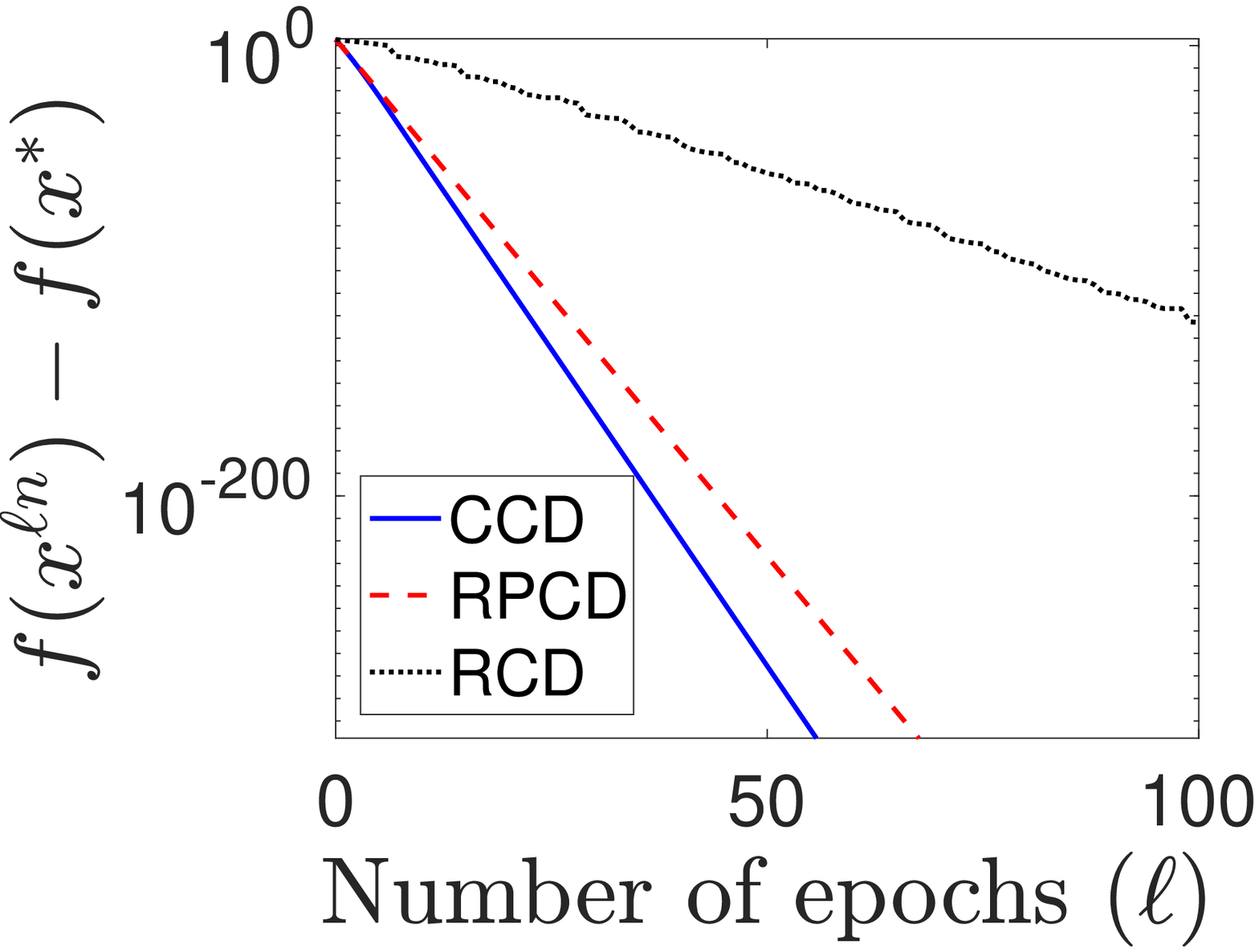}
  \end{minipage}
  \hfill
  \begin{minipage}[b]{0.32\textwidth}
    \includegraphics[width=\textwidth]{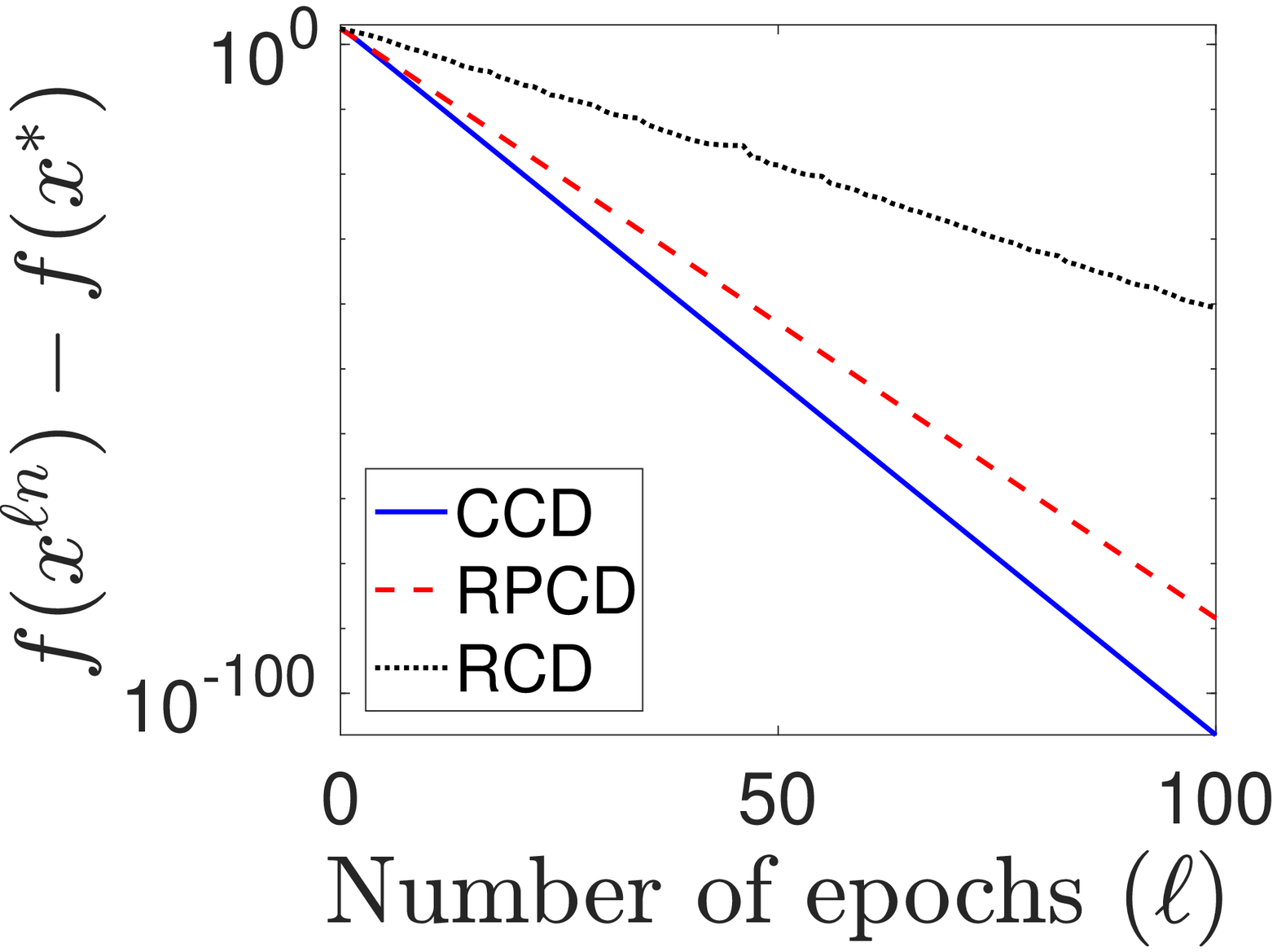}
  \end{minipage}
  \hfill
  \begin{minipage}[b]{0.32\textwidth}
    \includegraphics[width=\textwidth]{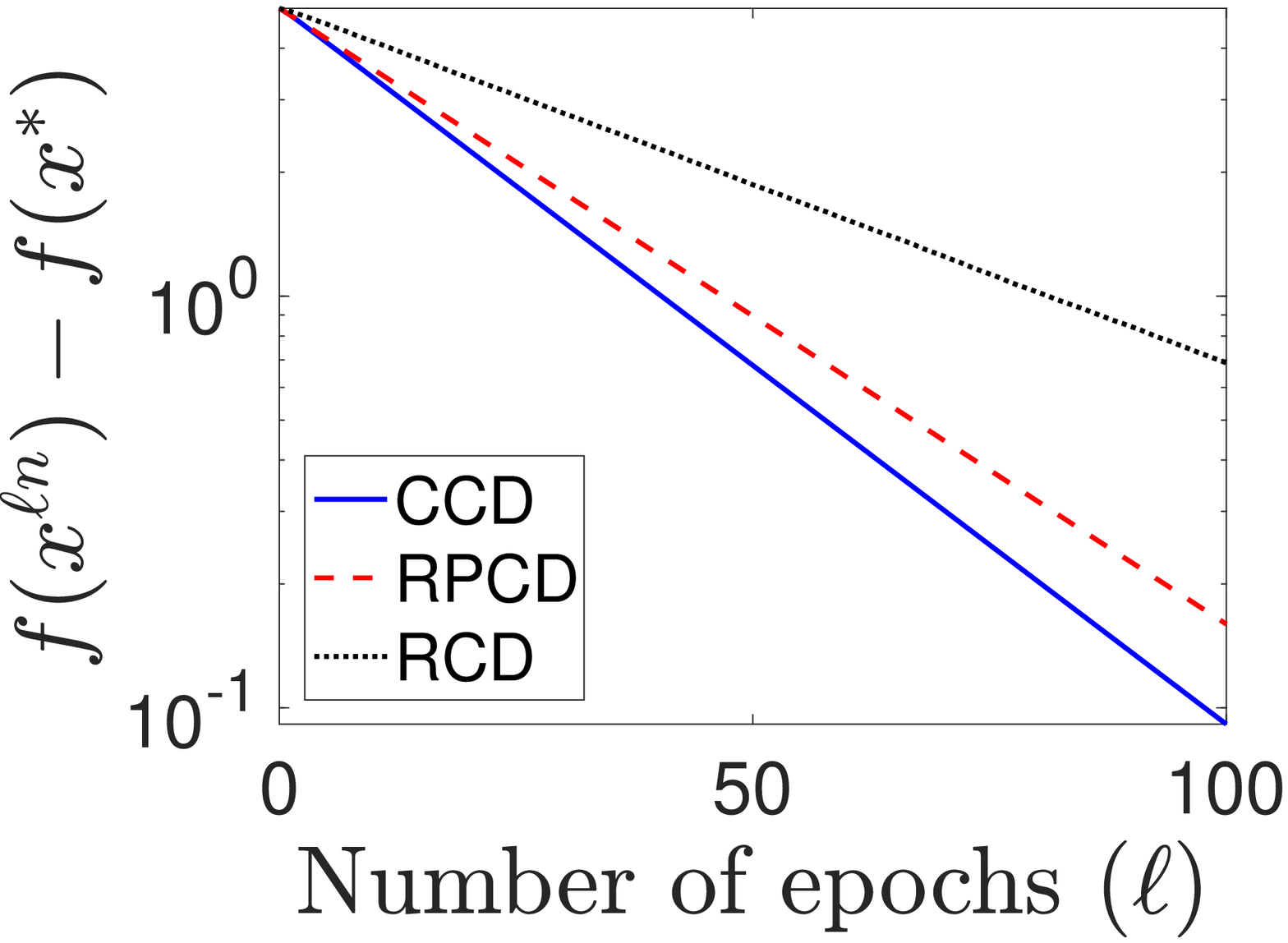}
  \end{minipage}
  \hfill
  \begin{minipage}[b]{0.32\textwidth}
    \includegraphics[width=\textwidth]{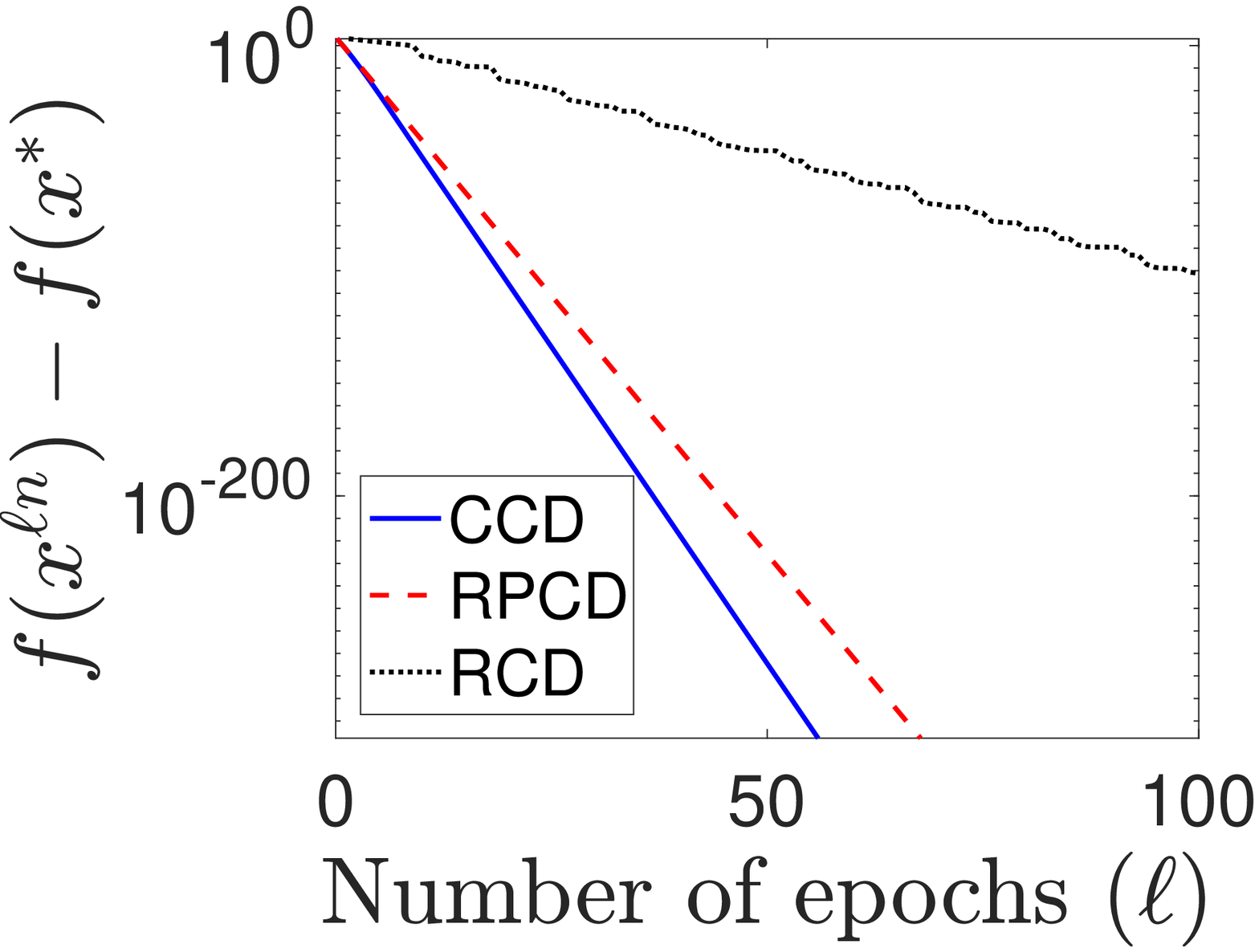}
  \end{minipage}
  \hfill
  \begin{minipage}[b]{0.32\textwidth}
    \includegraphics[width=\textwidth]{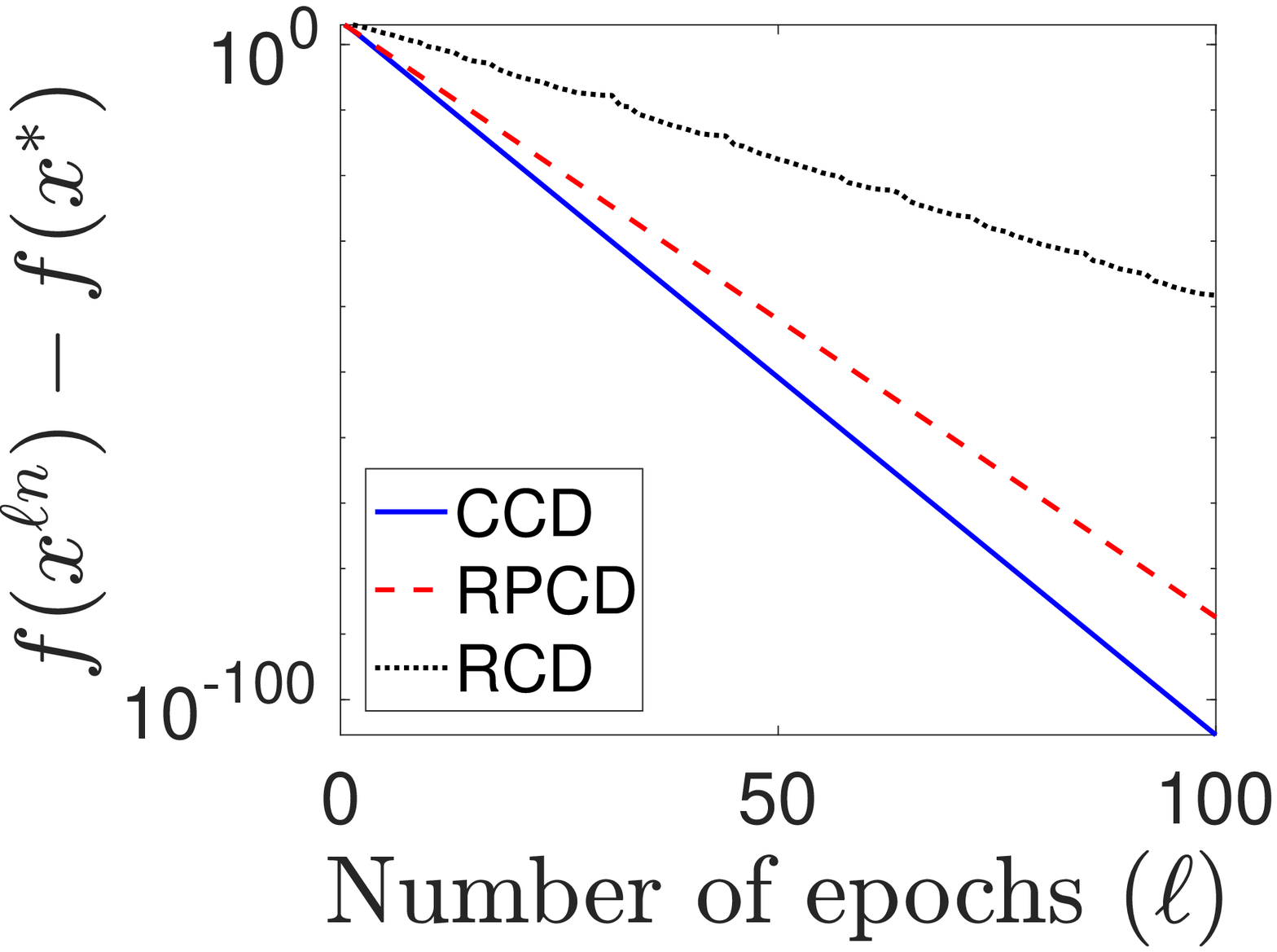}
  \end{minipage}
  \hfill
  \begin{minipage}[b]{0.32\textwidth}
    \includegraphics[width=\textwidth]{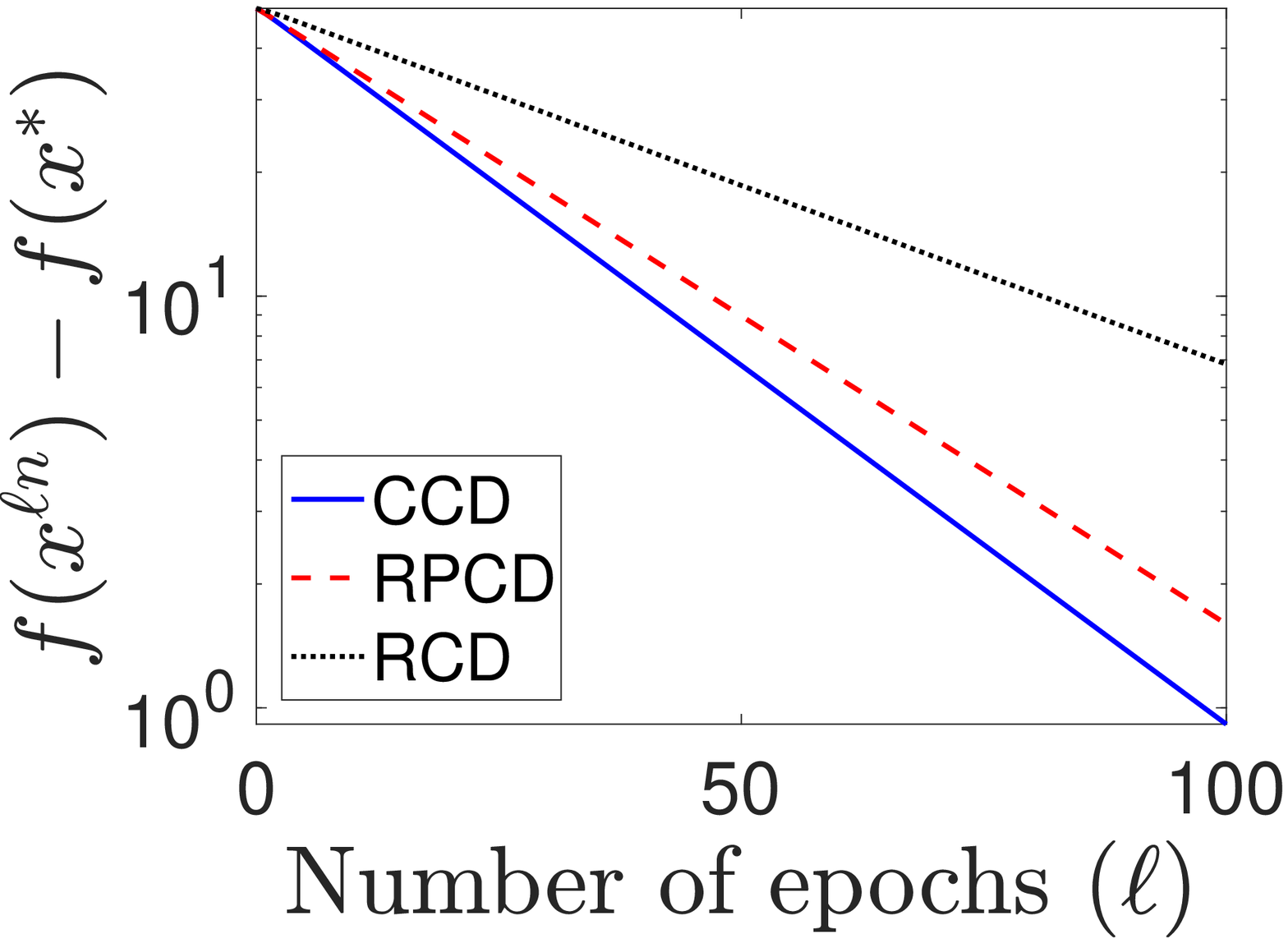}
  \end{minipage}
  \caption{CCD vs RPCD vs RCD with worst-case initialization for $n=1000$ (top row) and $n=10000$ (bottom row): $\alpha=\frac{0.01}{n-1}$ in the left column, $\alpha=\frac{0.50}{n-1}$ in the middle column, and $\alpha=\frac{0.99}{n-1}$ in the right column.}
  \label{fig:worst}
\end{figure}

\begin{figure}[!tbp]
  \centering
  \begin{minipage}[b]{0.32\textwidth}
    \includegraphics[width=\textwidth]{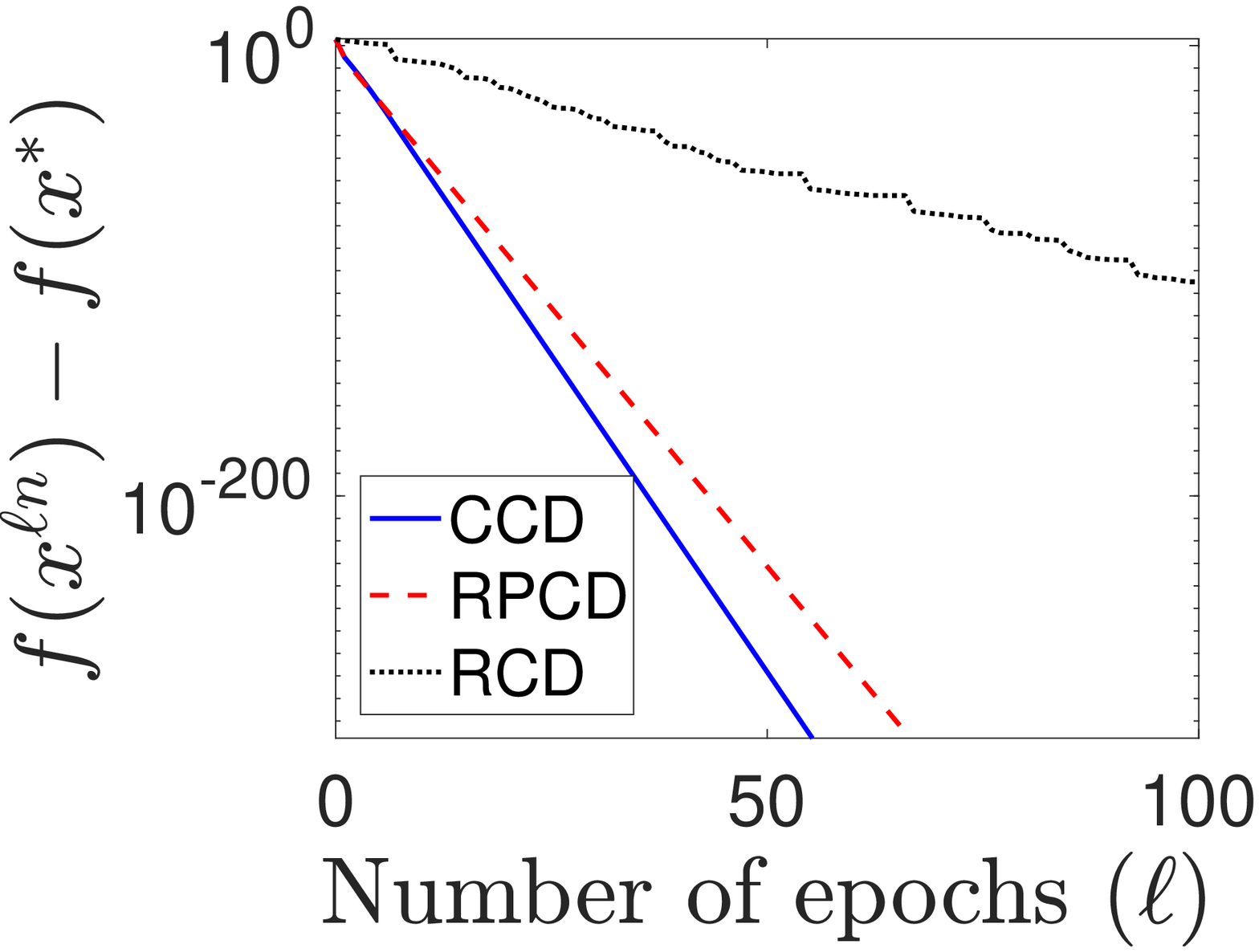}
  \end{minipage}
  \hfill
  \begin{minipage}[b]{0.32\textwidth}
    \includegraphics[width=\textwidth]{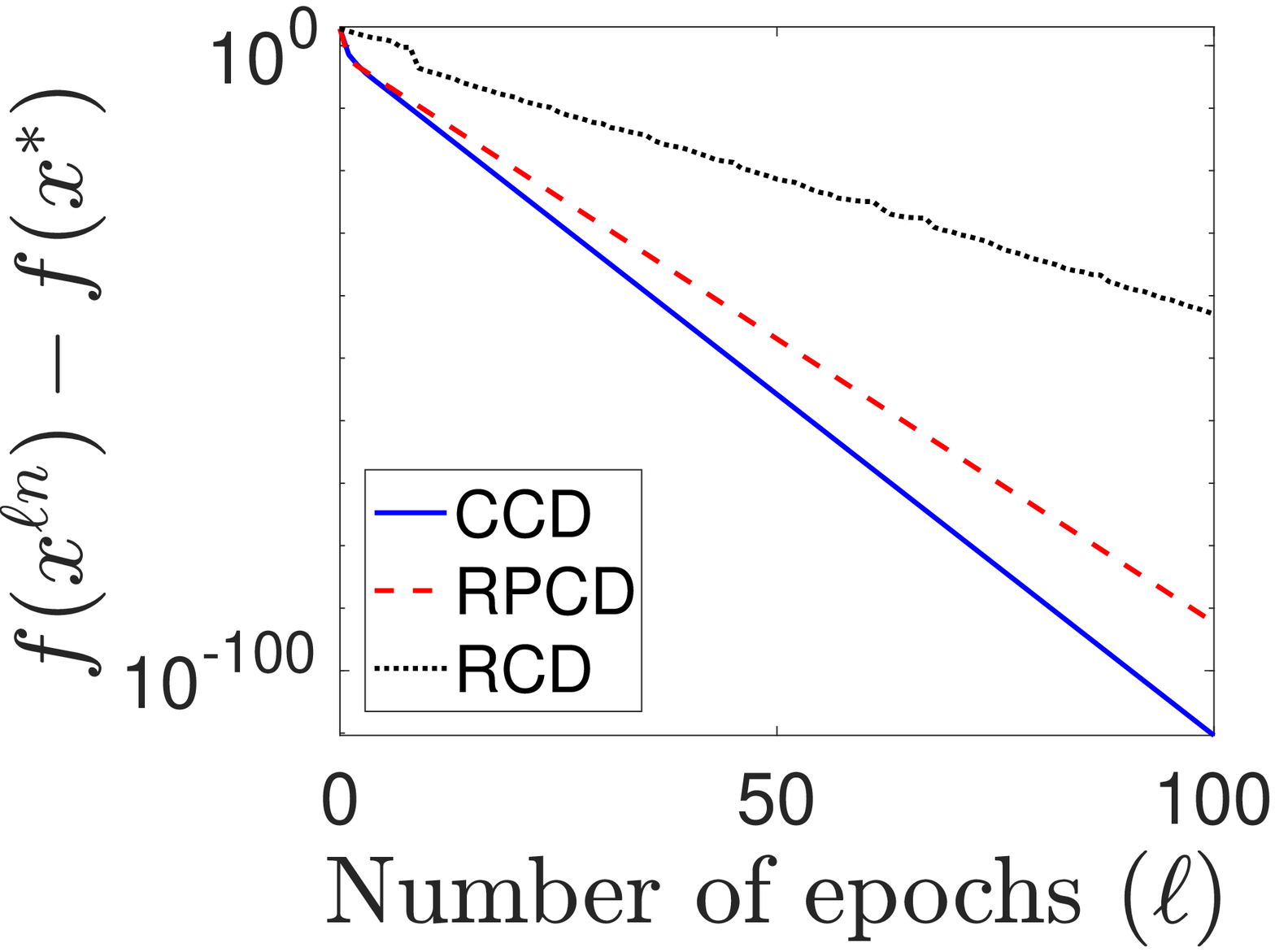}
  \end{minipage}
  \hfill
  \begin{minipage}[b]{0.32\textwidth}
    \includegraphics[width=\textwidth]{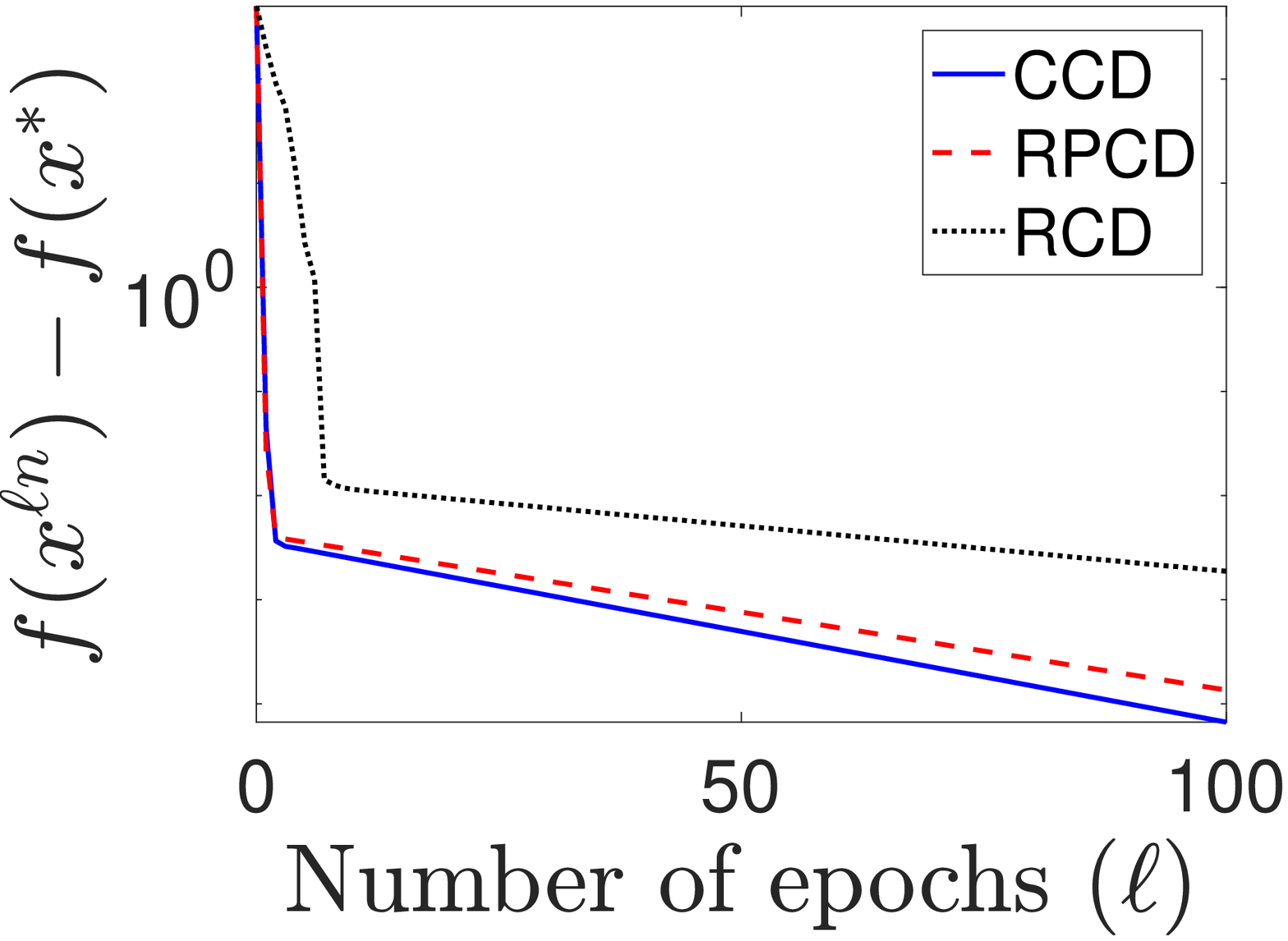}
  \end{minipage}
  \caption{CCD vs RPCD vs RCD with random initialization for $n=1000$: $\alpha=\frac{0.01}{n-1}$ (left figure), $\alpha=\frac{0.50}{n-1}$ (middle figure), and $\alpha=\frac{0.99}{n-1}$ (right figure).}
  \label{fig:random}
\end{figure}

\section{Conclusion}\label{sec:conclusion}
In this paper, we surveyed the known results on the performance of RPCD for special cases of strongly convex quadratic objectives and add to these results by presenting a class of convex quadratic problems with diagonally dominant Hessians. Using the distance of the expected iterates to the optimal solution as the convergence criterion, we compared the ratio between the performances of RPCD and RCD with respect to a parameter that represents the extent of diagonal dominance. We illustrated that as the Hessian matrix becomes more diagonally dominant, this ratio goes to infinity, whereas as it gets smaller it goes to a constant in the interval $[3/2, \, e-1)$. We also showed that CCD outperforms both RPCD and RCD for this class of problems. When expected distance of the iterates or expected function value of the iterates is used as the convergence criterion, we presented that the worst-case convergence rate bounds derived for RPCD are tighter compared to the ones for RCD. This is in accordance with our first set of results, i.e., when distance of the expected iterates is used as the convergence criterion. Computational experiments validate our theoretical results, which fill a gap between the theoretical guarantees for RPCD and its empirical performance.

\bibliographystyle{plain}
\bibliography{rpcd,nips_ref}

\appendix

\section{Proof of Lemma \ref{lemm-gamma}} \label{app:lemm-gamma-proof}
Applying Lemma \ref{lem:wright_lemma} with $Q=\Gamma^{-1}$, where $\Gamma^{-1}$ is defined in \eqref{eqn-gamma-inv}, we get
\begin{align*}
	\gamma & = \frac{\sum_{j=0}^{n-2} (n-1-j) \alpha (1+\alpha)^j}{n(n-1)}  = \frac{\alpha}{n} \sum_{j=0}^{n-2} (1+\alpha)^j - \frac{\alpha}{n(n-1)} \sum_{j=0}^{n-2} j (1+\alpha)^j \\
		& = \frac{(1+\alpha)^{n-1}-1}{n} - \frac{(1+\alpha)^{n-1}}{n} + \frac{(1+\alpha)^n-1-\alpha}{\alpha n(n-1)} 
		 = \frac{(1+\alpha)^n-\alpha n-1}{\alpha n(n-1)},
\end{align*}
where the third equality follows by the following lemma. This completes the proof.

\begin{lemma}\label{lem-sum-rule} For any real scalar $\eta \neq 1$ and integer $k\geq 0$, we have $$\sum_{j=0}^k j \eta^j  = (k+1)\frac{\eta^{k+1}}{\eta-1} - \frac{(\eta^{k+1}-1)\eta}{(\eta-1)^2}.$$
\end{lemma}
\begin{proof}[Lemma \ref{lem-sum-rule}]
Consider the cumulative sums  $u_k(\eta) :=  \sum_{j=0}^k \eta^j = \frac{\eta^{k+1}-1}{\eta - 1}$. It is easy to see that $ \sum_{j=0}^k j \eta^j = \eta u_k'(\eta) $ where $u_k'(\eta)$ is the derivative of $u_k(\eta)$. Differentiating the right-hand side of the formula for $u_k$ yields the result.
\end{proof}

\section{Proof of Proposition \ref{thm-monotonic-rpcd-speedup}}
\begin{proof}[Part (i)]: Defining $h(t,n) = \frac{ \left( 1+ \frac{t}{n-1} \right)^n - 1} {\frac{t}{n-1}}$, where $t\in(0,1)$ and $n\geq1$ is an integer, we have by the definition in \eqref{eq:s-definition} that $s(t,n) = \rho_1(t,n)/\rho_2(t,n)$, where
\begin{equation*}
	\rho_1(t,n) = -\log \left (1 - \frac{1-t}{n} h(t,n)  \right) \quad \mbox{and} \quad \rho_2(t,n) =  -n \log \left (1 - \frac{1-t}{n}  \right).
\end{equation*}
Throughout the rest of the proof, for simplicity, whenever the dependence of $h$, $\rho_1$ and $\rho_2$ on $n$ is clear, we will abbreviate them by $h(t)$, $\rho_1(t)$ and $\rho_2(t)$, respectively. Similarly, whenever the dependence on $t$ is also clear, we will abbreviate them by $h$, $\rho_1$ and $\rho_2$, respectively. In order to prove statement {\em (i)} of Proposition \ref{thm-monotonic-rpcd-speedup}, it suffices to show that the partial derivative satisfies
\begin{equation*}
	\partial_t s(t,n) = \frac{\partial_t (\rho_1) \rho_2 - \rho_1 \partial_t(\rho_2)}{\rho_2^2} < 0,
\end{equation*}
for all $t \in (0,1)$. This holds if and only if
\begin{equation} \label{eq:prop1-step1}
	\frac{\partial_t (\rho_1)}{\rho_1} <  \frac{\partial_t (\rho_2)}{\rho_2} \quad \iff \quad \partial_t \left( \log  \rho_1 \right) < \partial_t \left( \log \rho_2 \right),
\end{equation}
for all $t \in (0,1)$, where we used the fact that $\rho_1$ and $\rho_2$ are positive for $t\in (0,1)$. We can compute these partial derivatives in the right-hand side as follows
\begin{equation*}
	\partial_t \left( \log  \rho_1 \right) = \frac{1}{\rho_1} \partial_t(\rho_1) = \frac{-1}{\rho_1} \left( \frac{1}{1 - \frac{1-t}{n} h(t)} \right) \left( \frac{h(t) + h'(t) (t-1)}{n} \right),
\end{equation*}
and similarly 
\begin{equation*}
	\partial_t \left( \log  \rho_2 \right) = \frac{1}{\rho_2} \partial_t(\rho_2) = \frac{-1}{\rho_2} \left( \frac{1}{1 - \frac{1-t}{n}} \right).
\end{equation*}
Hence, in order to prove \eqref{eq:prop1-step1}, it is sufficient to show that
\begin{equation*}
	\frac{1}{\rho_1} \left( \frac{1}{1 - \frac{1-t}{n} h(t)} \right)  g(t)  > \frac{1}{\rho_2} \left( \frac{1}{1 - \frac{1-t}{n}} \right), \quad \mbox{where}  \quad q(t):=\frac{h(t) + h'(t) (t-1)}{n},
\end{equation*}
which, after inserting the formulas for $\rho_1$ and $\rho_2$, is equivalent to
\begin{equation}\label{eq:prop1-step2}
	-n \log\left(1 - \frac{1-t}{n}\right) \left(1 - \frac{1-t}{n}\right) q(t) > -\log\left(1 - \frac{1-t}{n}h\right) \left(1 - \frac{1-t}{n}h\right),
\end{equation}
for $t\in (0,1)$. The main ingredients to prove this inequality is to approximate the non-linear functions $q$ and $h$ with piecewise linear functions, which are easier to deal with, in other words, linearizing $q$ and $h$ above leads to simpler expressions for the derivatives of both sides of this inequality. In order to approximate $q$, we first write a binomial expansion for $h(t)$ as follows
\begin{equation*}
	h(t) = \frac{ \left( 1+ \frac{t}{n-1} \right)^n - 1} {\frac{t}{n-1}} = \sum_{i=1}^n \binom ni \left( \frac{t}{n-1} \right)^{i-1}.
\end{equation*}
This implies that $q(t)$ is of the form $ q(t) = \frac{1}{2} + \frac{2}{3}t + \sum_{j=2}^{n-1} c_j t^j $, where $c_2 > 0$ and $c_j \geq 0$, for all $j\in\{3,\dots,n-1\}$. Therefore, the first and second derivatives of $g$ are positive over $t\in(0,1)$ and $g$ is strictly convex. We then consider linearizations of $q(t)$ at $t=0$ and $t=1$, which are given by
\begin{equation*}
	q_0(t) =  \frac{1}{2} + \frac{2}{3}t \quad \mbox{and} \quad q_1(t) = \frac{h(1) - 2(n-1)(1-t)}{n}.
\end{equation*}
(Note that in the special case $n=2$, $q(t)$ is linear so that $q_0(t) = q_1(t)$ for all $t$. However, for $n>2$, $q_0 \neq q_1$). In particular, it can be checked that $q_0(\hat{t}) = q_1(\hat{t})$, for $\hat{t} = 1 - \frac{6h(1) - 7n}{4(2n-3)}$. Since $q(t)$ is convex,
\begin{equation}\label{def-g-lower-bound}
	q(t) \geq \underline{q}(t) = \max\left(q_0(t), q_1(t)\right) = 
    \begin{cases}
    	q_0(t),  &\mbox{if } t\in [0, \hat{t}), \\
		q_1(t),  & \mbox{if } t \in [\hat{t},1].
	\end{cases}
\end{equation}
The right-hand side of $\eqref{eq:prop1-step2}$ is of the form
	\beq\label{def-z} z(t) = -\log\left (y(t)\right) y(t) = E(y(t)), \quad \mbox{where} \quad y(t)=1 - \frac{1-t}{n}h, \quad E(y)= -\log(y) y.\eeq
As $h$ is convex, we have the bounds
	\beq\label{def-y-bar}\hub(t) = (1-t) h(0) + t h(1) \geq h(t) \quad \mbox{and} \quad y(t) \geq \yub(t)=1 - \frac{1-t}{n}\hub, \quad t\in (0,1).
    \eeq
Using the facts that the function $E(\cdot)$ has a maximum of $1/e$ over the interval $[0,1]$ and is strictly decreasing over the interval $(1/e,1]$, it follows from \eqref{def-y-bar} that
  \beq\label{def-z-upper-bound} E(y(t)) =z(t) \leq \overline{z}(t) := \begin{cases} 
  E(\bar{y}(t))          & \mbox{if} \quad \yub \in (1/e,1] \iff t \in (t_*,1]	 \\
  	1/e 			& \mbox{if} \quad \yub \in [0,1/e] \iff t\in  [0,t_*]
  	\end{cases} \eeq
where $t_*$ is the largest $t\in(0,1)$ such that $\yub(t)=1/e$ and admits the formula 
	$$ t_* = -\frac{1}{2}\frac{2n-h(1)}{h(1)-n} + \frac{1}{2}\sqrt{ 
    \left(\frac{2n-h(1)}{h(1)-n}\right)^2 + \frac{4}{e}\frac{n}{h(1)-n}}.$$
Combining the lower bound \eqref{def-g-lower-bound} on $q(t)$ and the upper bound \eqref{def-z-upper-bound} on $z(t)$, a sufficient condition for \eqref{eq:prop1-step2} is to show that the following relaxed inequality holds
 	\begin{equation}\label{ineq-to-prove}
	-n \log\left(1 - \frac{1-t}{n}\right) \left(1 - \frac{1-t}{n}\right) \qlb(t) - \zub (t) > 0, \quad \mbox{for all} \quad t \in (0,1).
\end{equation}
The left-hand side is a piecewise continuously differentiable function (pieces defined by the intervals $[0,\hat{t}]$, $(\hat{t}, t_*]$ and $(t_*,1]$)) and it is positive at $t=0$. The rest of the proof is about showing that the left-hand side in \eqref{ineq-to-prove} stays positive for $t\in(0,1)$, this is achieved by computing and lower bounding the first order derivatives of the left-hand side. The details are skipped due to space considerations and follows from standard calculus techniques.

\end{proof}

\begin{proof}[Part (ii)]:
Since $\lim_{t\to 0^+} \rho_2(t) = -n \log(1-1/n)$, whereas $\lim_{t\to 0^+} \rho_1(t) = -\log(1-h(0)/n)=\infty$ as $h(0)=n$, we obtain $\lim_{t\to 0^+} s(t,n) = \lim_{t\to 0} \left(\rho_1(t)/\rho_2(t)\right) = \infty$.
\end{proof}

\begin{proof}[Part (iii)]:
We observe that
$	g(n) =  \lim_{t\to 1^-} \frac{\rho_1(t)}{\rho_2(t)} =  \lim_{t\to 1^-} \frac{\rho'_1(t)}{\rho'_2(t)},
$
since $\lim_{t\to 1^-} \rho_1(t) = \lim_{t\to 1^-} \rho_2(t) = 0$. The derivatives of $\rho_1(t)$ and $\rho_2(t)$ with respect to $t$ are given by
\begin{equation*}
	\rho'_1(t) = - \frac{h(t) + h'(t) (t-1)}{n-(1-t)h(t)} \quad \mbox{and} \quad \rho'_2(t) = -\frac{n}{n-(1-t)}.
\end{equation*}
Therefore, we obtain
\begin{equation*}
	g(n) = \lim_{t\to 1^-} \frac{\frac{h(t) + h'(t) (t-1)}{n-(1-t)h(t)}}{\frac{n}{n-(1-t)}} = \frac{h(1)}{n} = \left( 1+ \frac{1}{n-1} \right)^{n-1} + \frac{1}{n} - 1.
\end{equation*}
In order to show that $g(n)$ is strictly increasing in $n$, consider the extension of $g$ to the positive real line, i.e., consider the function $\bar{g}(z) = \left(1+\frac{1}{z}\right)^{z} + \frac{1}{z+1} - 1$, where $z\geq0$. Taking its derivative with respect to $z$, we get
\begin{equation*}
	\bar{g}'(z) = \left( \log\left(1+\frac{1}{z}\right) - \frac{1}{z+1} \right) \left(1+\frac{1}{z}\right)^{z} - \frac{1}{(z+1)^2}.
\end{equation*}
Using the lower bounds $\log(1+y) \geq \frac{2y}{2+y}$ for $y\geq0$ and $(1+1/y)^y \geq 2$ for $y\geq1$, we obtain
\begin{equation*}
	\bar{g}'(z) \geq 2 \left( \frac{2}{2z+1} - \frac{1}{z+1} \right) - \frac{1}{(z+1)^2} = \frac{1}{(z+1)(z+1/2)}- \frac{1}{(z+1)^2} > 0,
\end{equation*}
for any $z\geq1$. Consequently, $g(n)$ is strictly increasing in $n\geq2$. Furthermore, it follows directly from the definition that $g(2) = 3/2$ and since $\lim_{n\to\infty}(1+1/n)^n = e$, we get $\lim_{n\to\infty} g(n) = e-1$. This completes the proof of part $(iii)$.
\end{proof}

\section{Proof of Proposition \ref{thm:ccd}} \label{app:ccd}
The proof of $\rho(\rpcdE)<\rho(\rcdE)^n$ follows by Proposition \ref{thm-monotonic-rpcd-speedup}, hence is omitted. Since the off-diagonal entries of $A$ are nonpositive and $A$ is a positive definite matrix, then it follows by \cite[Theorem 4.12]{ccd_vs_rcd} that $\rho(\ccd) \leq \frac{1-\mu}{1+\mu} = 1-\frac{2\mu}{1+\mu}$, where $\mu=1-(n-1)\alpha$. On the other hand, from \eqref{eq-rho-RPCD}, we have $\rho(\rpcdE) = 1 - \mu \frac{(1+\alpha)^n  - 1}{n\alpha}$. Hence, in order to show that $\rho(\ccd) < \rho(\rpcdE)$, for all $\alpha\in(1,1/(n-1))$ and $n\geq2$, it suffices to show
\begin{equation*}
	\frac{2}{1+\mu} > \frac{(1+\alpha)^n  - 1}{n\alpha} \quad \iff \quad \frac{1}{1-\frac{(n-1)\alpha}{2}} > \frac{(1+\alpha)^n  - 1}{n\alpha}.
\end{equation*}
Since $\alpha\in(1,1/(n-1))$, it is sufficient to show that
\begin{equation}\label{eq:ccd_req}
	n\alpha > \left( 1-\frac{(n-1)\alpha}{2} \right) \left( (1+\alpha)^n  - 1 \right).
\end{equation}
Using the Binomial expansion $(1+\alpha)^n = \sum_{j=0}^n {n \choose j} \alpha^j$, we get
\begin{align*}
\small
    \left( 1-\frac{(n-1)\alpha}{2} \right) \left( (1+\alpha)^n  - 1 \right) & = \sum_{j=1}^n {n \choose j} \alpha^j - \frac{n-1}{2} \sum_{j=1}^n {n \choose j} \alpha^{j+1} \\
    	& < \sum_{j=1}^n {n \choose j} \alpha^j - \frac{n-1}{2} \sum_{j=1}^{n-1} {n \choose j} \alpha^{j+1} \\
    	& = n\alpha + \sum_{j=2}^n \left( {n \choose j} - \frac{n-1}{2} {n \choose j-1} \right) \alpha^j,
\end{align*}
where the inequality follows since we omit the last term of the second sum and the last equality follows by peeling out the first entry of the first sum. We can observe that
\begin{equation*}
	{n \choose j} - \frac{n-1}{2} {n \choose j-1} = \left( \frac{n+1-j}{j} - \frac{n-1}{2} \right) {n \choose j-1} = \left( \frac{(n+1)(2-j)}{2j} \right) {n \choose j-1} \leq 0,
\end{equation*}
for all $j\in\{2,\dots,n\}$. This proves \eqref{eq:ccd_req}, which concludes the proof.

\section{Proof of Proposition \ref{lemma-rcd-upper bound}} \label{app:rcd-bound}
RCD iterations can be written (by \eqref{eq:rcd-definition}) as follows
\begin{equation*}
	\xr^{k+1} = \left( I - e_{i_k} e_{i_k}^T A\right) \xr^k,
\end{equation*}
where $i_k$ is drawn uniformly at random from the set $\{1,2,\dots,n\}$. Letting $\E_k$ denote the expectation with respect to $i_k$ given $x_k$ and taking norm squares of both sides, we obtain
\begin{align*}
	\E_k \|\xr^{k+1}\|^2 & = (\xr^k)^T \, \E_k \left[\left( I - A^T e_{i_k} e_{i_k}^T \right) \left( I - e_{i_k} e_{i_k}^T A\right)\right] \xr^k \\
		& =  (\xr^k)^T \left( \frac{1}{n} \sum_{i=1}^n \left( I - A^T e_{i} e_{i}^T - e_{i} e_{i}^T A + A^T e_{i} e_{i}^T A \right) \right) \xr^k \\
		& = (\xr^k)^T \left( I - \frac{2A}{n} + \frac{A^2}{n} \right) \xr^k \leq \norm{Q} \|\xr^k\|^2 \mbox{ with } Q := I - \frac{2A}{n} + \frac{A^2}{n},
\end{align*}
where we used the fact that $A=A^T$ and $\sum_{i=1}^n e_ie_i^T = I$. Using this recursion and noting that $x^*=0$, we get
\begin{equation} \label{eq-example-rcd-decay}
	\E \| \xr^{k+1} - x^* \|^2 \leq \| Q \|^k \, \| x^0 - x^*\|^2.
\end{equation}
The eigenvalues of $Q$ are of the form $1-2\lambda/n + \lambda^2/n$, where $\lambda$ is an eigenvalue of $A$. Since $Q$ is symmetric and $A$ has only two distinct eigenvalues that are equal to $\mu = (1-\alpha (n-1))$ and $L=1+\alpha$, we obtain
\begin{align}
	\| Q \| & = \max \{1-2\mu/n + \mu^2/n, 1-2L/n + L^2/n\}
		 = 1-2\mu/n + \mu^2/n. \label{eq-Q-norm}
\end{align}
Using \eqref{eq-Q-norm} in \eqref{eq-example-rcd-decay} concludes the proof of \eqref{eq:upper bound-rcd norm square}. The proof of \eqref{eq:upper bound-rcd subopt} can be done by following similar lines to the above proof as follows
\begin{align*}
	f(\xr^{k+1}) & = (\xr^k)^T \, \E_k \left[\left( I - A^T e_{i_k} e_{i_k}^T \right) A \left( I - e_{i_k} e_{i_k}^T A\right)\right] \xr^k \\
		& =  (\xr^k)^T \, \E_k \left[ A - A^T e_{i_k} e_{i_k}^T A - A e_{i_k} e_{i_k}^T A + A^T e_{i_k} e_{i_k}^T A e_{i_k} e_{i_k}^T A \right] \xr^k \\
		& =  (\xr^k)^T \, \E_k \left[ A - A e_{i_k} e_{i_k}^T A \right] \xr^k \\
		& =  (\xr^k)^T \, \left( A - \frac{A^2}{n} \right) \xr^k 
		 \leq \norm{I - \frac{A}{n}} f(\xr^k) 
		 = \left( 1 - \frac{\mu}{n} \right) f(\xr^k),
\end{align*}
where in the third equality, we use the fact that $A=A^T$ and $e_i^T A e_i=1$, for all $i\in[n]$, and in the fourth equality, we use $\sum_{i=1}^n e_ie_i^T = I$, respectively. This concludes the proof.

\vspace{-.5cm}
\section{Proof of Proposition \ref{theo-subopt-rpcd}} \label{app:pot6}
RPCD iterations can be written (by \eqref{eq:rpcd}) as follows
\begin{equation*}
	\xp^{(\ell+1) n} = P_{\pi_\ell} \ccd P_{\pi_\ell}^T \, \xp^{\ell n}. 
\end{equation*}
Considering improvement sequence $\Ly_2$, this yields
\begin{equation*}
	\E_\ell \| \xp^{(\ell+1) n} \|^2 = (\xp^{\ell n})^T \E_P [ P \ccd^T \ccd P^T ] \xp^{\ell n} \leq \| S \| \| \xp^{\ell n} \|^2,
\end{equation*}
where $S=\E_P [P \ccd^T \ccd P^T]$. Using this recursion, we obtain
\begin{equation*}
	\E \| \xp^{\ell n} \|^2 \leq \| S \|^\ell \big\| \xp^0 \big\|^2.
\end{equation*}
The contraction factor $\| S \|$ can be computed by applying Lemma \ref{lem:wright_lemma} with $Q = \ccd^T \ccd$, which yields
\begin{equation} \label{eq:expected iterate square for rpcd}
	S=\E_P[P \ccd^T \ccd P^T]= \tau_1 I + \tau_2 \bfone\bfone^T, 
\end{equation}
where 
\begin{equation*}
	\tau_2 = \frac{\bfone^T \ccd^T \ccd \bfone - \trace(\ccd^T \ccd)}{n(n-1)} \quad \mbox{and} \quad \tau_1 = \frac{\trace (\ccd^T \ccd)}{n} - \tau_2.
\end{equation*}
Since $S$ is a symmetric matrix, we have $\| S \| = \rho(S)$. Furthermore, we can observe that $\ccd^T \ccd$ has strictly positive entries both in its diagonals and off-diagonals, consequently we have $S>0$. Then, by Perron-Frobenius Theorem \cite[Lemma 2.8]{varga2009matrix}, we have
\begin{equation} \label{eq:rho-S}
	\| S \| = \rho(S) = \tau_1 + n\tau_2 = \frac{1}{n} \bfone^T S \bfone.
\end{equation}
In order to compute \eqref{eq:rho-S}, we first compute the matrix $\ccd$ as follows
\begin{equation}\label{eq-ccd iter matrix for ex} 
	\ccd = I - \Gamma^{-1}A = \begin{cases}
            \alpha \left((1+\alpha)^{i-1}- (1+\alpha)^{i-j}\right), & \mbox{if} \quad i \geq j, \\
            \alpha(1+\alpha)^{i-1}, & \mbox{if} \quad i<j.
        \end{cases}
\end{equation}
Combining \eqref{eq:rho-S} and \eqref{eq-ccd iter matrix for ex}, we obtain
\begin{equation*}
	\| S \| = \frac{1}{n}\bfone^T \ccd^T \ccd \bfone = \frac{1}{n} \| \ccd \bfone \|^2  = \frac{1}{n} \sum_{i=1}^n \left( (\ccd\bfone)_i \right)^2,
\end{equation*}
where
\begin{equation}\label{eq-rowsum}
	(\ccd\bfone)_i = 1 - \mu (1+\alpha)^{i-1}.
\end{equation}
This yields
\begin{align*}
	\| S \| & = \frac{1}{n} \sum_{i=1}^n \left( 1 - 2 \mu (1+\alpha)^{i-1} + \mu^2 (1+\alpha)^{2(i-1)} \right) = 1 - \frac{2\mu}{n} \left( \frac{(1+\alpha)^n - 1}{\alpha} \right) +  \frac{\mu^2}{n} \left( \frac{(1+\alpha)^{2n} - 1}{\alpha(\alpha+2)}\right),
\end{align*}
which proves \eqref{eq:upper bound-rpcd norm square}.


We next prove the results regarding the function suboptimality in \eqref{eq:rpcd-f-upper-bound}. To this end, we consider the expected function sub-optimality (note that $f(x^*)=0$), which yields
\begin{align*}
	\E_\ell f(\xp^{(\ell+1) n}) & = (\xp^{\ell n})^T \E_P [P \ccd^T P^T A P \ccd P^T] \xp^{\ell n} \\
    	& = (\xp^{\ell n})^T \E_P [P \ccd^T A \ccd P^T] \xp^{\ell n} \\
    	& \leq \| \E_P [A^{-1/2} P \ccd^T A \ccd P^T A^{-1/2}] \| \| A^{1/2} \xp^{\ell n} \|^2 \\
    	& = \| \E_P [A^{-1/2} P \ccd^T A \ccd P^T A^{-1/2}] \| \, f(\xp^{\ell n}) \\
    	& = \| \E_P [P A^{-1/2} \ccd^T A \ccd A^{-1/2} P^T] \| \, f(\xp^{\ell n}) \\
        & = \| G \| \, f(\xp^{\ell n}),
\end{align*}
where $G:=\E_P [PA^{-1/2} \ccd^TA \ccd A^{-1/2}P^T]$ and the equalities follow since $A$ and $A^{-1/2}$ are symmetric matrices. It can be shown that $A^{1/2} \ccd A^{-1/2}$ is a non-negative matrix, hence applying Lemma \ref{lem:wright_lemma} to the matrix $Q = A^{-1/2} \ccd^T A \ccd A^{-1/2}$, it can be shown (similar to the previous proof) that
\begin{equation} \label{eq:norm-G}
	\|G\| = \rho(G) = \frac{1}{n} \| A^{1/2} \ccd A^{-1/2} \bfone \|^2 = \frac{1}{n} \| \bfone - A^{1/2} \Gamma^{-1} A^{1/2} \bfone \|^2,
\end{equation}
where $A^{1/2} = \gamma I - \sigma \bfone \bfone^T$ with $\gamma = \sqrt{1+\alpha}$ and $\sigma=(\gamma-\sqrt{\mu})/n$. This yields
$A^{1/2} \bfone = (\gamma-n\sigma) \bfone = \sqrt{\mu} \bfone$. Multiplying both sides of the above equality by $\Gamma^{-1}$ from the left, we obtain
\begin{equation} \label{eq:app6-step1}
	\Gamma^{-1} A^{1/2} \bfone = \sqrt{\mu} \, c,
\end{equation}
where it follows from \eqref{eqn-gamma-inv} that
\begin{equation*}
	c =
    \begin{bmatrix}
		1 \\ 1+\alpha \\ 1+\alpha+\alpha(1+\alpha) \\ \vdots \\ 1+\alpha+\alpha(1+\alpha)+\dots+\alpha(1+\alpha)^{n-2}
	\end{bmatrix}
    =
    \begin{bmatrix}
		1 \\ 1+\alpha \\ (1+\alpha)^2 \\ \vdots \\ (1+\alpha)^{n-1}
	\end{bmatrix}.
\end{equation*}
Multiplying \eqref{eq:app6-step1} from the left by $A^{1/2}$, we get
\begin{equation} \label{eq:app6-step2}
	A^{1/2} \Gamma^{-1} A^{1/2} \bfone = \sqrt{\mu} \left( \gamma c - \sigma \norm{c}_1 \bfone \right), \quad \mbox{where} \quad \norm{c}_1 = \frac{(1+\alpha)^n-1}{\alpha}.
\end{equation}
Using \eqref{eq:app6-step2} in \eqref{eq:norm-G}, we obtain
\begin{align}
	\norm{G} & = \frac{1}{n} \sum_{i=1}^n \left( 1 - \sqrt{\mu} \left( \gamma c_i - \sigma \norm{c}_1 \right) \right)^2 \nonumber= 1 - \frac{2\sqrt{\mu}}{n} \sum_{i=1}^n \left( \gamma c_i - \sigma \norm{c}_1 \right) + \frac{\mu}{n} \sum_{i=1}^n \left( \gamma c_i - \sigma \norm{c}_1 \right)^2  \nonumber\\
    	& = 1 - \frac{2\sqrt{\mu}}{n} \left( \gamma - n \sigma \right) \norm{c}_1 + \frac{\mu}{n} \sum_{i=1}^n \left( \gamma^2 c_i^2 - 2\gamma \sigma \norm{c}_1 c_i + \sigma^2 \norm{c}_1^2 \right)  \nonumber\\
    	& = 1 - \frac{2\mu}{n} \norm{c}_1 + \frac{\mu}{n} \left( \gamma^2 \norm{c}_2^2 - 2\gamma \sigma \norm{c}_1^2 + n \sigma^2 \norm{c}_1^2 \right), \label{eq:app6-step3}
\end{align}
where
\begin{equation*}
	\norm{c}_2^2 = \frac{(1+\alpha)^{2n}-1}{\alpha(\alpha+2)} \quad \mbox{and} \quad \norm{c}_1^2 = \frac{(1+\alpha)^{2n}-2(1+\alpha)^n+1}{\alpha^2}.
\end{equation*}
Modifying the terms in \eqref{eq:app6-step3}, we get
\begin{align*}
	\norm{G} & = 1 - \frac{2\mu}{n} \norm{c}_1 + \frac{\mu}{n} \left( \gamma^2 \norm{c}_2^2 - \gamma \sigma \norm{c}_1^2 + \sigma ( n\sigma-\gamma) \norm{c}_1^2 \right) \\
    	& = 1 - \frac{2\mu}{n} \norm{c}_1 + \frac{\mu}{n} \left( (1+\alpha) \norm{c}_2^2 - \frac{1+\alpha-(1-\alpha(n-1))}{n} \norm{c}_1^2 \right) \\
    	& = 1 - \frac{2\mu}{n} \norm{c}_1 + \frac{\mu}{n} \left( (1+\alpha) \norm{c}_2^2 - \alpha \norm{c}_1^2 \right) \\
    	& = 1 - \frac{2\mu}{n} \norm{c}_1 + \frac{\mu}{n} \left( (1+\alpha) \frac{(1+\alpha)^{2n}-1}{\alpha(\alpha+2)} - \frac{(1+\alpha)^{2n}-2(1+\alpha)^n+1}{\alpha} \right) \\
    	& = 1 - \frac{\mu}{n} \left( \frac{(1+\alpha)^{2n}-1}{\alpha(\alpha+2)} \right),
\end{align*}
which concludes the proof of Proposition \ref{theo-subopt-rpcd}.

\end{document}